\documentclass[11pt]{amsart}

\usepackage{graphicx,amssymb,amsfonts,epsfig,amsthm,a4,amsmath,url,amscd,mathrsfs,xspace,hyperref}
\usepackage{graphicx}
\usepackage[francais,english]{babel}
\usepackage{inputenc}
\usepackage[T1]{fontenc}
\usepackage[margin=1.45in,centering]{geometry}
\usepackage{csquotes}
\usepackage{color}
\usepackage{enumerate}
\usepackage{enumitem}

  \newcommand{\R}{\ensuremath{\mathbb{R}}}%
  \newcommand{\Z}{\ensuremath{\mathbb{Z}}}%
	\newcommand{\Q}{\ensuremath{\mathbb{Q}}}%
  \newcommand{\N}{\ensuremath{\mathbb{N}}}%
	\newcommand{\F}{\ensuremath{\mathcal{F}}}%

        \newcommand{\V}{\ensuremath{\mathcal{V}}}%
\renewcommand{\fg}{\ensuremath{\operatorname{FG}}}

				\newcommand{\X}{\ensuremath{\mathcal{X}}}%
                \renewcommand{\P}{\ensuremath{\mathcal{P}}}%

                \renewcommand{\Pi}{\ensuremath{\mathcal{P}_\perp}}%

        \renewcommand{\H}{\ensuremath{\mathcal{H}}}%
				\renewcommand{\L}{\ensuremath{\mathcal{L}}}%
                \newcommand{\K}{\ensuremath{\mathcal{K}}}%

  \newcommand{\Sym}{\ensuremath{\operatorname{Sym}}}%
    \newcommand{\sym}{\ensuremath{\operatorname{Sym}}}%

  \newcommand{\supp}{\ensuremath{\operatorname{Supp}}}%
	\newcommand{\fix}{\ensuremath{\operatorname{Fix}}}%
  \newcommand{\sub}{\ensuremath{\operatorname{Sub}}}%
  \newcommand{\homeo}{\ensuremath{\operatorname{Homeo}}}%
    \newcommand{\aut}{\ensuremath{\operatorname{Aut}}}%
    \newcommand{\st}{\ensuremath{\operatorname{St}}}%
    \newcommand{\rist}{\ensuremath{\operatorname{R}}}%
    \newcommand{\iet}{\ensuremath{\operatorname{IET}}}%
    \newcommand{\rk}{\ensuremath{\operatorname{rk}}}%
  \newcommand{\vol}{\ensuremath{\operatorname{vol}}}%
        \newcommand{\leud}{\ensuremath{\operatorname{dim}_{\operatorname{LE}}}}%

    \newcommand{\tg}{\ensuremath{\mathsf{t}}}%
   \newcommand{\org}{\ensuremath{\mathsf{o}}}%

  \newtheorem{thm}{Theorem}[section]
  \newtheorem{main thm}{Theorem}
  \newtheorem{prop}[thm]{Proposition}
    \newtheorem{prop-def}[thm]{Proposition-Definition}

  \newtheorem{cor}[thm]{Corollary}
   \newtheorem{lem}[thm]{Lemma}
      \newtheorem*{lem-spe}{Lemma}

\theoremstyle{definition}
  \newtheorem{defin}[thm]{Definition}

	 \newtheorem{notation}[thm]{Notation}

\theoremstyle{remark}
  \newtheorem{remark}[thm]{Remark}

  \newtheorem{example}[thm]{Example}

\begin{document}

\title[A commutator lemma for confined subgroups]{A commutator lemma for confined subgroups and applications to groups acting on rooted trees}



\author{Adrien Le Boudec}
\address{CNRS, UMPA - ENS Lyon, 46 all\'ee d'Italie, 69364 Lyon, France}
\email{adrien.le-boudec@ens-lyon.fr}
\thanks{ALB supported by ANR-19-CE40-0008-01 AODynG}

\author{Nicol\'as Matte Bon}
\address{
	CNRS,
	Institut Camille Jordan (ICJ, UMR CNRS 5208),
	Universit\'e de Lyon,
	43 blvd.\ du 11 novembre 1918,	69622 Villeurbanne,	France}

\email{mattebon@math.univ-lyon1.fr}
\thanks{NMB was supported in part by the SNF grant 200020\_169106}

\thanks{This work was supported by the LABEX MILYON (ANR-10-LABX-0070) of Université de Lyon, within the program "France 2030" (ANR-11-IDEX-0007) operated by the French National Research Agency (ANR).}


\date{19 April 2023}

\dedicatory{}

\begin{abstract}
A subgroup $H$ of a group $G$ is confined if the $G$-orbit of $H$ under conjugation is bounded away from the trivial subgroup in the space $\sub(G)$ of subgroups of $G$. We prove a commutator lemma for confined subgroups. For groups of homeomorphisms, this provides the exact analogue for confined subgroups (hence in particular for URSs) of the classical commutator lemma for normal subgroups: if $G$ is a group of homeomorphisms of a Hausdorff space $X$ and $H$ is a confined subgroup of $G$, then $H$ contains the derived subgroup of the rigid stabilizer of some open subset of $X$.

We apply this commutator lemma in the setting of groups acting on rooted trees. We prove a theorem describing the structure of URSs of weakly branch groups and of their non-topologically free minimal actions. Among the applications of these results, we show: 1) if $G$ is a finitely generated branch group, the $G$-action on $\partial T$ has the smallest possible growth among all faithful $G$-actions; 2) if $G$ is a finitely generated branch group, then every embedding from $G$ into a group of homeomorphisms of strongly bounded type (e.g. a bounded automaton group) must be spatially realized;  3) if $G$ is a finitely generated weakly branch group, then $G$ does not embed into the group IET of interval exchange transformations.

\end{abstract}

\maketitle

\setcounter{tocdepth}{2}

\section{Introduction}

Given a group $G$, we denote by $\sub(G)$ the space of subgroups of $G$, endowed with the induced topology from the set $\left\lbrace 0,1 \right\rbrace ^G$ of all subsets of $G$. The group $G$ acts on $\sub(G)$ by conjugation. A subgroup $H$ of $G$ is \textbf{confined} if the closure of the $G$-orbit of $H$ in $\sub(G)$ does not contain the trivial subgroup $\left\lbrace 1\right\rbrace $. By definition of the topology in $\sub(G)$, this is equivalent to saying that there exists a finite subset $P$ of non-trivial elements of $G$ such that $g H g^{-1} \cap P$ is not empty for every $g \in G$. A subset $P$ with this property is called a \textbf{confining subset} for $(H,G)$. When $G$ is  finitely generated, yet another equivalent definition is that  a subgroup $H$ of $G$ is confined if the Schreier graph associated to $H$ does not contain copies of arbitrarily large balls of the Cayley graph of $G$ (see Section \ref{s-preliminaries} for the precise formulation).   

A related notion is the notion of uniformly recurrent subgroup (URS) introduced by Glasner and Weiss \cite{GW-urs}. By definition a URS is a (non-empty) closed $G$-invariant subspace $\H \subseteq \sub(G)$ that is minimal (with respect to inclusion) among closed $G$-invariant subspaces of $\sub(G)$. Equivalently, $\H$ is a URS is for every $H \in \H$, the orbit closure of $H$ is equal to $\H$. By extension we also say that subgroup $H$ of $G$ is a URS if the orbit closure of $H$ is minimal. Clearly, if $H$ is non-trivial and $H$ is a URS, then $H$ is confined. Conversely, if $H$ is confined, then the orbit closure of $H$ contains a non-trivial URS by Zorn's lemma. 

Historically, confined subgroups were first systematically studied in the setting of locally finite groups, in connection with the study of ideals in group algebras \cite{Sehg-Zal-Alt93,Zaless94,Hart-Zal-97,Leinen-Puglisi-02Isr,Leinen-Puglisi-02Pac,Leinen-Puglisi-03}. More recently, confined subgroups and URSs played an important role in the study of ideals in reduced group $C^*$-algebras \cite{Kal-Ken,Ken-URS,LBMB-sub-dyn}. Beyond these connections with group algebras, confined subgroups and URSs were shown to be useful tools to establish rigidity results about various families of countable groups appearing as groups of homeomorphisms \cite{LBMB-sub-dyn,MB-graph-germs}.

Given a group $G$ acting on a set $\Omega$ and a subset $\Sigma$ of $\Omega$, we denote by $\rist_G(\Sigma)$ the \textbf{rigid stabilizer} of $\Sigma$ in $G$, that is the pointwise stabilizer in $G$ of $\Omega \setminus \Sigma$. An elementary observation in the study of normal subgroups is the following: if a group $G$ acts faithfully on a set $\Omega$, and if $\Sigma$ is a subset of $\Omega$ and $N$ is a normal subgroup of $G$ such that there exist $g \in N$ such that $\Sigma$ and $g(\Sigma)$ are disjoint, then $N$ contains the derived subgroup $\rist_G(\Sigma)'$ of $\rist_G(\Sigma)$. This classical trick goes back at least to Higman \cite{Higman54}, and is the common denominator of many proofs of simplicity. It is sometimes referred to as the \enquote{commutator lemma} or \enquote{double commutator lemma}, as its proof consists of a suitable commutator manipulation (see Lemma \ref{lem-comm-lemma-norm-abs}). One particular setting in which the commutator lemma has been used extensively is the case of groups of homeomorphisms. In this setting it says that if $G$ is a group of homeomorphisms of a Hausdorff space $X$ and $N$ is a non-trivial normal subgroup of $G$, then $N$ contains $\rist_G(U)'$ for some non-empty open subset $U\subset X$. 

In this article we prove the exact extension of this statement to confined subgroups.

\begin{thm} \label{thm-doublecomm-conf-homeo-intro}
	Let $G$ be a group of homeomorphisms of a Hausdorff space $X$. If $H$ is a confined subgroup of $G$ (e.g.\ if $H$ is a non-trivial URS of $G$), then there exists a non-empty open subset $U\subset X$ such that $H$ contains $\rist_G(U)'$. 
\end{thm}

Although  the statement of Theorem \ref{thm-doublecomm-conf-homeo-intro} does not contain any assumption on the group  $G\le \homeo(X)$, its conclusion non-trivial only when $\rist_G(U) \neq 1$ for every non-empty open subset $U$. A group of homeomorphisms $G\le \homeo(X)$ satisfying this condition is said to be \textbf{micro-supported}.

We actually derive Theorem \ref{thm-doublecomm-conf-homeo-intro} from a more constructive version of the commutator lemma for confined subgroups in the abstract setting of group actions on sets (Theorem \ref{thm-conf-intro}). That result has the advantage of being applicable outside of the realm of micro-supported groups of homeomorphisms, and has other applications. In  \cite{LBMB-ht} we use this result to establish a connection between confined subgroups of a group $G$ and highly transitive actions of $G$.


Micro-supported group of homeomorphisms have been the subject of a program that aims at understanding, given a micro-supported subgroup $G\le \homeo(X)$, how much the space $X$ and the action of $G$ on $X$ are intrinsically associated to $G$ as an abstract group. An instance of this program is given by the reconstruction results of Rubin \cite{Rub-tr}, which provide various (mild) sufficient conditions under which two micro-supported faithful actions of the same group $G$ by homeomorphisms on two spaces $X, Y$ must be conjugate, or equivalently under which an isomorphism between two micro-supported groups $G, H$  must be implemented by a homeomorphism of the underlying spaces.  Similar results in this spirit were obtained in \cite{Whittaker, Filip}  in the setting of groups of homeomorphisms and diffeomorphisms of manifolds, in \cite{Dye-I} in the setting of measured group actions, or also in \cite{Lav-Nek} and \cite{GPS,Medynets, Mat-SFT} for groups of homeomorphisms of the Cantor set. 

A considerably less understood problem, suggested for example by Rubin in \cite[p.493]{Rub-tr}, is the following. Given a micro-supported group $G\le \homeo(X)$, find natural conditions under which an action of $G$ on another space $Y$ (not necessarily micro-supported) must be related to the action of $G$ on $X$, for instance by the existence of a continuous equivariant map from $Y$ to $X$ (or at least to a close kin, such as its space of closed subsets $\F(X)$). Another version of this problem is to find natural conditions on a group of homeomorphisms $H\le \homeo(Y)$ under which any embedding of $G$ into $H$ must automatically give rise to such a map. 

One direction of application of Theorem \ref{thm-doublecomm-conf-homeo-intro} is that the study of confined subgroups is a key tool to study these problems. To understand the connection, recall that an action of a group $G$ on a compact space $Y$ is \textbf{topologically free} if the germ stabilizer $G_y^0$ is trivial for every $y \in Y$. Here $G_y^0$ is the set of elements of $G$ acting trivially on a neighbourhood of $y$. Conversely we say that the action of $G$ on $Y$ is \textbf{topologically nowhere free} if  $G_y^0$ is non-trivial for every $y \in Y$. If the action is minimal, this is equivalent to the fact that the action is not topologically free. Observe that topologically nowhere free is a much weaker condition than micro-supported. One readily checks that compactness of $Y$ implies that if the action of $G$ on $Y$ is topologically nowhere free, then $G_y^0$ is a  confined subgroup of $G$ for every $y\in Y$. Therefore, given  any micro-supported group $G\le \homeo(X)$, Theorem \ref{thm-doublecomm-conf-homeo-intro} establishes a connection between the natural $G$-action on $X$ and all topologically nowhere free $G$-actions.

A version of Theorem \ref{thm-doublecomm-conf-homeo-intro} has been previously obtained in \cite[Theorem 3.10]{LBMB-sub-dyn} under a strong additional assumption on the dynamics of the action of $G$ on $X$, namely that the action of $G$ on $X$ is extremely proximal, and this assumption was weakened in \cite{MB-graph-germs} to the assumption that the action of $G$ on $X$ is proximal. Recall that an action of a group $G$ on a topological space $X$ is \textbf{proximal} if for every pair of points $x, y$, there exists a net $(g_i)$ in $G$ such that $g_i(x)$ and $g_i(y)$  converge to the same point $z\in X$.  These weaker versions of Theorem \ref{thm-doublecomm-conf-homeo-intro} were used  among other things in \cite{LBMB-sub-dyn} to study $C^*$-simplicity in combination with \cite{Kal-Ken,Ken-URS} and to prove a rigidity theorem for  non-topologically free minimal actions of certain micro-supported groups (including Thompson's groups), and in \cite{MB-graph-germs} to prove structure theorems for embeddings of topological full groups of pseudogroups over the Cantor set into other groups of homeomorphisms.

The main novelty of Theorem \ref{thm-doublecomm-conf-homeo-intro} over these previous versions is that it holds without any assumption on the action of $G$ on $X$, and this opens the way to applications to  broader classes of micro-supported groups. One important class of actions in topological dynamics that is at the opposite of the class of proximal actions is the class of profinite actions, i.e.\ inverse limit of finite actions. This class admits the following characterization: an action of a group $G$ on a compact metrizable space $X$ is profinite if and only if there exists a  locally finite rooted tree $T$ and an action of $G$ by automorphisms on $T$ such that the action of $G$ on the boundary $\partial T$ is conjugate to the action of $G$ on $X$ (see Proposition \ref{p-odometers}). Following Grigorchuk, a group $G$ is called a \textbf{weakly branch group} if $G$ admits a profinite action on a metrizable space that is faithful, minimal and micro-supported. Equivalently, $G$ is a weakly branch group if there exists a locally finite rooted tree on which $G$ acts faithfully such that the action of $G$ on $\partial T$ is minimal and micro-supported.  This class of groups is rich and includes well-studied examples such as Grigorchuk's groups from \cite{Gri-growth}, the Basilica group and the Gupta-Sidki groups. See \cite{BGS-branch} for a survey and for additional examples. In this setting  Lavreniuk and Nekrashevych proved the following reconstruction result: if $G$ admits a faithful weakly branch action on two trees $T,T'$, then the actions on $\partial T$ and $\partial T'$ are conjugate \cite{Lav-Nek}. Hence for a weakly branch group the $G$-space $\partial T$ is well-defined and canonically attached to $G$ (although the tree $T$ itself is not).

In Sections \ref{sec-trees-min-actions} and \ref{sec-trees-graphs-actions} of the article we focus on applications of our results on confined subgroups within the class of weakly branch groups. In particular we prove:

\begin{enumerate}[label=\Roman*)]
	\item  a structure theorem about URSs of weakly branch groups (Theorem \ref{t-wb-urs}), and a structure theorem about their non-topologically free minimal actions (Theorem \ref{t-wb-structure-lsc}).
	\item for finitely generated branch groups, the growth of every faithful $G$-action on a set   is bounded below by the growth of the action of $G$ on $\partial T$ (Theorem \ref{thm-branch-wobb}). The definition of growth of an action can be found shortly below in the introduction, or in \S \ref{subsec-growth}.
	
		\item finitely generated weakly branch groups cannot embed into the group IET of interval exchange transformations (Theorem \ref{t-wb-iet}).
	\item a spatial realization theorem about embeddings of a finitely generated branch group into a class of groups of homeomorphisms of the Cantor space, which includes bounded automata groups and topological full groups of Cantor minimal systems (Theorem \ref{t-strongly-bounded}). 
\end{enumerate}

In the remainder of this introduction we give a more detailed overview of these results.

\addtocontents{toc}{\protect\setcounter{tocdepth}{1}}

\subsection*{URSs of weakly branch groups and applications to growth of actions.} We apply the commutator lemma for confined subgroups to study URSs of weakly branch groups.  Simple constructions carried out in \S \ref{subsec-2-constructions} provide two natural ways to obtain families of URSs in a weakly branch group $G$. The flexibility of these constructions provide in particular continuously many distinct URSs in $G$. Our main result on URSs (Theorem \ref{t-wb-urs}) provides structural information about URSs of a weakly branch group. This result implies that a lot of information can be recovered on a URS $H$ from its set of fixed points $\fix(H)$ in $\partial T$: the set $\fix(H)$ varies continuously with $H$, and one can find a partition of the complement of $\fix(H)$ in $\partial T$ into cylinders subsets such that $H$ contains the derived subgroup of the rigid stabilizer of each of these cylinders subsets, and moreover this partition also varies continuously with $H$. We refer to Section \ref{sec-trees-min-actions} for details. In the setting of IRSs, a similar statement was obtained by Zheng in  \cite{Zheng-irs}, where it is deduced from a commutator lemma for IRSs of groups of homeomorphisms. In the special case of finitary regular branch groups, a similar IRS statement was also previously obtained in \cite{Fer-Tot}. 

We give an application of Theorem \ref{t-wb-urs} to the study of the  growth of actions of finitely generated branch groups. Recall that if $G = \left\langle  S  \right\rangle $ is a finitely generated group and $X$ is a $G$-set, the growth of the action of $G$ on $X$ is the function 

\[\vol_{G,S,X}(n)= \sup_{x\in X} |B_S(n) \cdot x|,\] where $B_S(n)$ is the ball of radius $n$ in $G$ around the identity with respect to the word metric associated to $S$. For example the growth of the left action of $G$ on a coset space $G/H$ is the uniform growth of the Schreier graph associated to the subgroup $H$. Up to a natural equivalence relation, the function $\vol_{G,S,X}$ does not depend on $S$, and is denoted $\vol_{G,X}$. Clearly $\vol_{G,X}$ is bounded above by the growth $\vol_{G}$ of the group $G$ (which corresponds to the left action of $G$ on itself), but many finitely generated groups $G$ admit faithful actions whose growth is strictly smaller than the growth of $G$. For instance, non-abelian free groups admit faithful actions with linear growth. Thus given a finitely generated group, it is natural to ask how small the growth of its faithful actions can be. 

Various examples of Schreier graphs of weakly branch groups $G \leq \aut(T)$ associated to the action of $G$ on $\partial T$ have been studied in details in the literature \cite{Hanoi-GS,Bon,Vor,DDMN}, and play a role in the study of existence of free subgroups and of amenability of such groups \cite{Nek-free, JNS}.  In many of these examples,  the graphs of the action of $G$ on $\partial T$ are quite small and much easier to describe that the Cayley graph of $G$. In particular the growth  $\vol_{G,\partial T}$ is often slow, for instance it is polynomial if $G$ is a contracting group (see \cite{Nek-book}); but it can also be exponential, for instance for the groups from \cite{Sidki-Wilson}. 

An an application of our structure results on URSs of branch groups, we prove the following (see Theorem \ref{thm-branch-wobb}):


\begin{thm} \label{thm-branch-wobb-intro}
	Let $G \leq \aut(T)$ be a finitely generated  branch group. Then for every $G$-set $X$ on which $G$ acts faithfully, the growth $\vol_{G, X}$ satisfies $\vol_{G, X} \succeq \vol_{G, \partial T}$.  
\end{thm}

\subsection*{Non-topologically free actions of weakly branch groups and applications.} 
We also prove a structure theorem that applies to non-topologically free minimal actions of weakly branch groups (Theorem \ref{t-wb-structure-lsc}).  A consequence of this theorem is that every faithful and minimal action of a weakly branch group $G$ that is not topologically free admits as a factor a non-trivial closed $G$-invariant subspace of the space $\F(\partial T)$ of closed subsets of $\partial T$. In particular, it factors onto a non-trivial $G$-space on which the $G$-action is profinite (Corollary \ref{c-wb-actions}). The existence of a non-trivial profinite factor is a restrictive condition that has consequences on the dynamics of the action of $G$ on $X$. To put into context, recall that Frish, Tamuz and Vahidi Ferdowsi proved that every countable group with trivial FC-center (i.e.\ such that every non-trivial element has an infinite conjugacy class) admits a minimal and proximal faithful action  \cite{FTV}, that can be chosen to be topologically free \cite{GTWZ}. Corollary \ref{c-wb-actions} implies for instance that for weakly branch groups, \textit{every} faithful minimal and proximal action is topologically free. Similarly we also deduce that every faithful minimal and weakly mixing action is topologically free. See Corollary \ref{cor-wb-prox-wm}. 

In the case when the  growth of the action of  the finitely generated weakly branch group $G$ on $X$ is polynomially bounded, we  show that this profinite factor is  infinite (Theorem \ref{t-wb-polynomial}). To illustrate this result, we give an application related to the group $\iet$ of interval exchange transformations. An {interval exchange transformation} is a permutation of $\R/\Z$ with finitely many discontinuities, which coincides with a translation in restriction to each interval of continuity. While the dynamics of iterations of one element of IET is a well-studied topic, the study of more general groups that are capable of acting by interval exchanges has attracted attention recently. A central question is  to understand which finitely generated groups can embed in the group $\iet$. While a few obstructions to the existence of such embedding have been found \cite{Nov-disc, DFG-free, ExtAmen, DFG-solv, Cor-comm-pc}, this question remains not well-understood in general. For instance it is not known  if  non-abelian free groups can embed into $\iet$ (a question attributed to Katok in the literature \cite{DFG-free}), or if $\iet$ can contain infinite finitely generated periodic groups, or finitely generated groups with intermediate growth.  As an application of Theorem \ref{t-wb-polynomial}, we prove the following:

\begin{thm}\label{t-wb-iet-intro}
	Let $G \le \aut(T)$ be a finitely generated weakly branch group. Then $G$ does not admit any faithful action on $\R/\Z$ by interval exchange transformations.
\end{thm}

 We refer to Remark \ref{rmk-infgen-branch} for an example showing that the finite generation assumption is necessary in Theorem \ref{t-wb-iet}, and to Remark \ref{rmk-grig-group} for a more detailed discussion on the specific case of the Grigorchuk group.

Finally as another application, we explain how Theorem \ref{t-wb-structure-lsc} recovers the aforementioned result from \cite{Lav-Nek} which asserts that weakly branch groups admit only one faithful minimal and micro-supported action that is profinite (Corollary \ref{cor-reconstr-wb}). We also consider more general minimal and micro-supported actions of weakly branch groups, and show that every such $G$-space factors onto $\partial T$ and is  a highly proximal extension of $\partial T$ in the sense of Auslander and Glasner \cite{Ausl-Glasn} (Corollary \ref{cor-wb-micro-hp}).

\subsection*{Embeddings of weakly branch groups into other groups of homeomorphisms.} Theorem \ref{thm-doublecomm-conf-homeo-intro} provides a simple characterization of confined subgroups of weakly branch groups (Corollary \ref{cor-confined-wb}). We use this characterization to prove a rigidity result for embeddings of finitely generated branch groups $G$ into other groups of homeomorphisms. We show that if $H$ is a finitely generated group of homeomorphisms of a compact space $X$  such that the associated graphs of germs satisfy a suitable one-dimensionality condition, then every embedding $\rho \colon G\to H$ must be spatially realized in the sense that the induced action of $G$ on a natural subspace of $X$ factors onto the natural $G$-action on $\partial T$ (Theorem \ref{t-actions-leud}). This theorem applies for instance when the group $H$ belongs to the class of groups of homeomorphism of a Cantor set of strongly bounded type, defined in  \cite{JNS} (see Theorem \ref{t-strongly-bounded}). The class of groups of strongly bounded type includes for instance groups of bounded automorphisms of rooted trees \cite{Nek-free} (e.g. groups generated by finite-state bounded automata), which contain many well-studied examples of weakly branch groups acting on rooted trees, as well as other groups of homeomorphisms such as topological full groups of Cantor minimal systems. We refer to Section \ref{sec-trees-graphs-actions} for details.

\subsection*{Organization.} The article is organized as follows. In Section \ref{s-preliminaries} we set some notation and recall basic facts about URSs and semi-continuous maps. Section \ref{sec-confined} contains the proof of the commutator lemma for confined subgroups (Theorem \ref{thm-conf-intro}), and all the other sections of the article depend on this section. In Section \ref{sec-confined-homeo} we explain how to deduce Theorem \ref{thm-doublecomm-conf-homeo-intro} from Theorem \ref{thm-conf-intro}.

Sections \ref{sec-trees-min-actions} and \ref{sec-trees-graphs-actions} are devoted to groups acting on rooted trees. In Section \ref{sec-trees-min-actions} we apply the results of Section \ref{sec-confined} to obtain the structure theorem on URSs (Theorem \ref{t-wb-urs}) as well as a structure theorem about non-topologically free minimal actions (Theorem \ref{t-wb-structure-lsc}). This result is actually the core of this section, and Theorem \ref{t-wb-urs} is deduced from Theorem \ref{t-wb-structure-lsc}. At the end of Section \ref{sec-trees-min-actions} we also discuss some of the consequences of Theorems \ref{t-wb-urs} and \ref{t-wb-structure-lsc} that will be used in the last section.

In Section \ref{sec-trees-graphs-actions}  we prove all the other results mentioned in this introduction. The precise dependence between Section 6 and the rest of the article is that \S \ref{subsec-growth} and \S \ref{s-iet} depend on Theorem \ref{t-wb-structure-lsc}, while \S \ref{subsec-rigid-embed} only depends on Theorem \ref{thm-doublecomm-conf-homeo-intro}, so that the reader interested in \S \ref{subsec-rigid-embed} may fairly skip Section \ref{sec-trees-min-actions}.

\addtocontents{toc}{\protect\setcounter{tocdepth}{1}}

\let\oldtocsection=\tocsection

\let\oldtocsubsection=\tocsubsection

\let\oldtocsubsubsection=\tocsubsubsection

\renewcommand{\tocsection}[2]{\hspace{0em}\oldtocsection{#1}{#2}}
\renewcommand{\tocsubsection}[2]{\hspace{1em}\oldtocsubsection{#1}{#2}}
\renewcommand{\tocsubsubsection}[2]{\hspace{2em}\oldtocsubsubsection{#1}{#2}}

\tableofcontents

\section{Preliminaries} \label{s-preliminaries}

\subsection{Subgroups  and graphs associated to group actions} 
Let $G$ be a group acting on a topological space $X$. We denote by $G_x$ the stabiliser of $x\in X$ in $G$, and by $G^0_x$ the \textbf{germ-stabiliser} of $x$, i.e.\ elements $g\in G$ which fix pointwise a neighbourhood of $x$. Note that $G^0_x$ is a normal subgroup of $G_x$. The corresponding quotient $G_x/G^0_x$ is called the \textbf{isotropy group} (or \textbf{group of germs}) of $G$ at $x$. A point $x\in X$ is said to be \textbf{regular} if $G_x^0=G_x$. If $G$ is countable and  $X$ is a Baire space (e.g. if it is locally compact) then a standard application of Baire's theorem shows that the set of regular points is a dense $G_\delta$ subset of  $X$.

For $C\subset X$ we denote by $\fix_G(C)$ the pointwise stabilizer of $C$ in $G$.  The subgroup $\fix_G(X\setminus U)$, for $U\subset X$, is the \textbf{rigid stabiliser} of $U$, and is denoted $\rist_G(U)$. The action of $G$ on $X$ is \textbf{micro-supported} if $\rist_G(U)$ acts non-trivially on $X$ for every non-empty open subset $U\subset X$.

Given a subgroup $H$ of $G$, we denote by $\fix_X(H)$ the set of fixed points of $H$ in $X$. The (closed) \textbf{support} $\supp_X(H)$ of $H$ is defined as the closure  $\overline{X\setminus \fix_X(H)}\subset X$.

Assume now that $G$ is a finitely generated group acting on a set $X$, and let $S$ be a finite symmetric generating subset of $G$.  The \textbf{graph of the action} of $G$ on $X$ (with respect to $S$) is  is the graph $\Gamma(G, X)$ with vertex set $X$ and edge set $X\times S$, where  each edge $(x, s)$ connects the point $x$ with $sx$, and is labelled by the corresponding generator $s\in S$. Note that $\Gamma(G, X)$ is connected if and only if the action is transitive.  If $H\le G$ is a subgroup, the \textbf{Schreier graph} associated to $H$ is the graph $\Gamma(G, G/H)$ of the left action of $G$ on the coset space $G/H$. Given $x\in X$, the \textbf{orbital graph} $\Gamma(G, x)$ of $x$ is the graph $\Gamma(G, G/G_x)$. It is naturally isomorphic with the connected component of $\Gamma(G, X)$ containing $x$. Following the terminology of Nekrashevych \cite{Nek-frag}, the  \textbf{graph of germs} $\widetilde{\Gamma}(G, x)$ of $x$ is the graph $\Gamma(G, G/G_x^0)$. Note that elements of the vertex set of $\widetilde{\Gamma}(G, x)$ are in one-to-one correspondence with the set of \textbf{germs} of elements of $G$ at $x$, that is the equivalence classes of elements of $G$ where two elements are equivalent if they coincide on some neighbourhood of $x$.  Since $G_x^0$ is a subgroup of $G_x$, there is a natural map $\widetilde{\Gamma}(G, x) \to \Gamma(G, x)$, which is  a Galois cover with deck transformation group equal to the isotropy group $G_x/G^0_x$. Note that all the graphs defined here depend on the choice of the finite generating subset $S$ of $G$. However all the properties of these graphs that will be considered in the sequel are actually independent of $S$. This is why we omit $S$ in order to simplify the notation.

\subsection{Semi-continuous maps} \label{subsec-semi-conti}
 
When $Y$ is a locally compact space, we will denote by $\mathcal{F}(Y)$ the space of closed subsets of $Y$, endowed with the Chabauty topology, i.e.\ the topology generated by the sets  \[ \left\{C \in \mathcal{F}(Y)\, : \, C \cap K = \emptyset; \, C \cap U_i \neq \emptyset \, \, \text{for all $i$} \, \right\}, \] where $K \subset Y$ is compact and $U_1,\ldots,U_n \subset Y$ are open. The space $\mathcal{F}(Y)$ is compact. When $Y$ is discrete, this topology is the product topology on the set $\left\lbrace 0,1 \right\rbrace ^Y$ of all subsets of $Y$. When $G$ is a discrete group, the space $\sub(G)$ of subgroups of $G$ is closed in $\left\lbrace 0,1 \right\rbrace ^G$, and hence is a compact space.

Given a  space $X$, a map $\varphi: X \to \mathcal{F}(Y)$ is \textbf{upper semi-continuous} if for every compact space $K \subseteq Y$, the set of points $x \in X$ such that $\varphi(x) \cap K = \emptyset$ is open in $X$. This is equivalent to asking that whenever $(x_i)$ is a net in $X$ converging to $x$ such that $(\varphi(x_i))$ converges to $F$, one has $F \subseteq \varphi(x)$. Also $\varphi$ is \textbf{lower semi-continuous} if for every open subset $U \subseteq Y$, the set of points $x \in X$ such that $\varphi(x) \cap U \neq \emptyset$ is open in $X$; or equivalently if for every net $(x_i)$ converging to $x$ and such that $(\varphi(x_i))$ converges to $F$, one has $\varphi(x) \subseteq F$.

If a locally compact space $Y$ is second countable and $X$ is a Baire space (e.g. if $X$ is compact), then for every upper or lower semi-continuous map $\varphi: X \to \mathcal{F}(Y)$ the set of continuity points of $\varphi$ is a dense $G_\delta$ subset of $X$ (see e.g. \cite[Th.~VII]{Kuratowski1928}).

\begin{lem} \label{lem-compos-fix-lsc-usc}
Let $G$ be a  group, $X$ a compact $G$-space, and $Y$ a locally compact $G$-space. Let $\Phi\colon X\to \sub(G)$ be a lower semi-continuous map, and for $x\in X$ let $\fix_Y(\Phi(x))$ be the set of fixed points of $\Phi(x)$ in $Y$. Then the map $X\to \F(Y), x\mapsto \fix_Y(\Phi(x))$, is upper semi-continuous.
\end{lem}

\begin{proof}
If $K$ is a compact subset of $Y$ and $x \in X$ is such that $\Phi(x)$ fixes no point in $K$, we have to show that there is a neighbourhood of $x$ in $X$ consisting of points still having this property. For all $y \in K$ there exist an open subset $U_y$ of $Y$ and $g_y \in \Phi(x)$ such that $g_y(U_y) \cap U_y = \emptyset$. By compactness $K$ can be covered by finitely many $U_{y_1},\ldots,U_{y_n}$.  By lower semi-continuity the set $\{x'\in X \colon g_{y_1}, \ldots, g_{y_n}\in \Phi(x')\} $ is an open neighbourhood of $x$ with the desired property.
\end{proof}

\subsection{URSs and confined subgroups} \label{subsec-URS-conf}

Recall that a URS $\H$ of a group  $G$ is a closed and minimal $G$-invariant subspace of $\sub(G)$. If $H$ is a subgroup of $G$ having only finitely many conjugates $g_1 H g_1^{-1}, \ldots, g_n H g_n^{-1}$, then $\H = \left\lbrace g_1 H g_1^{-1}, \ldots, g_n H g_n^{-1} \right\rbrace $ is a banal example of a URS. Such a URS is called a \textbf{finite URS}. Note that by minimality and compactness, if $\H$ is a URS that is not finite, then $\H$ has no isolated points. By extension we also say that a subgroup $H$ of $G$ is a URS if the conjugacy class closure of $H$ in $\sub(G)$ is minimal. We will always use letters $\H$ for subsets of $\sub(G)$ and $H$ for subgroups of $G$, so that this there is no possible confusion.

Every minimal action of $G$ on a compact space $X$ gives rise to a URS $\mathcal{S}_{G}(X)$, called the \textbf{stabilizer URS} associated to $X$, which is the only minimal $G$-invariant closed subset in the closure of the set of subgroups $\left\lbrace G_x :  x \in X\right\rbrace $ \cite{GW-urs}. The action of $G$ on $X$ is topologically free if and only if $\mathcal{S}_{G}(X) = \left\lbrace \left\lbrace 1 \right\rbrace \right\rbrace $.

The inclusion between subgroups naturally induces a relation $\preccurlyeq$ on $\F(\sub(G))$, defined by $\H \preccurlyeq \K$ if there exists $H \in \H$ and $K \in \K$ such that $H \leq K$. It is not very hard to check that $\preccurlyeq$ is a partial order when restricted to URSs of $G$ \cite[Cor.\ 2.15]{LBMB-sub-dyn}. 

Recall that a subgroup $H$ of $G$ is \textbf{confined} if the closure of the $G$-orbit of $H$ under conjugation in $\sub(G)$ does not contain the trivial subgroup. This is equivalent to saying that there exists a finite subset $P$ of non-trivial elements of $G$ such that $g H g^{-1} \cap P$ is not empty for every $g \in G$, or equivalently such that the coset space $G/H$ is the union of the set of fixed points of elements of $P$ for the left action of $G$ on $G/H$. When $G = \left\langle S \right\rangle $ is a finitely generated group, a subgroup $H$ of $G$ is \textit{not} confined if and only if the Schreier graph $\Gamma(G, G/H)$ associated to $H$ contains isomorphic copies (as labelled graphs) of arbitrarily large balls of the Cayley graph of $G$.

\begin{remark}
The terminology ``confined subgroup'' was first used in the special case of locally finite groups \cite{Hart-Zal-97} with a slightly different definition: the confining set $P$ is required there to be a \emph{subgroup} of $G$ with the identity removed. The two definitions are equivalent when $G$ is locally finite.
\end{remark}

\subsection{Minimal actions and finite index subgroups}
In the sequel we will invoke the following basic lemma.


\begin{lem} \label{l-minimal-fi}
	Let $G$ be a group and $X$ be a minimal compact $G$-space. If $H$ be a finite index subgroup of $G$, then the closed minimal $H$-invariant subspaces of $X$ form a finite clopen partition of $X$.
\end{lem}

\begin{proof}
	Assume first that $H$ is normal. Let $X_1$ be a closed minimal $H$-invariant subspace of $X$, and $X_1, \ldots X_k$ be the collection of its distinct $G$-translates. Then each $X_i$ is $H$-invariant and minimal, so in particular $X_1,\ldots , X_k$ must be disjoint. Moreover by minimality of the $G$-action we must have $X=X_1\sqcup \cdots \sqcup X_k$, so it follows that $X_i$ is indeed clopen. 
	
	Assume now that $H$ is not necessarily normal, and choose a subgroup $K\le H$ that is normal and of finite index in $G$. By the previous paragraph we have a decomposition $X=X_1'\sqcup \cdots \sqcup X_r'$ into clopen minimal $K$-invariant subsets. Then $H$ acts on the set $\{X_1',\ldots, X'_r\}$. If $k$ is the number of orbits for this action and if $X_i$ is the union of the sets $X_j'$ in each orbit, $i=1, \ldots, k$, then $X_1, \ldots, X_k$ satisfy the conclusion.
\end{proof}

\section{Commutator lemma for confined subgroups} \label{sec-confined}


The goal of this section is to prove the commutator lemma for confined subgroups (Theorem \ref{thm-conf-intro}).  The proof is of purely group-theoretical nature, and is an elaboration of the arguments of \cite[Proposition 3.8]{LBMB-sub-dyn}.

\subsection{Confined subgroups}


The following definition generalizes the notion of a confined subgroup $H \leq G$ to the case where $H,G$ are two subgroups (not necessarily contained in each other) of an ambient group $L$. 

\begin{defin}
	Let $H,G$ be two subgroups of a group $L$. We say that $H$ is \textbf{confined} by $G$ if there exists a finite subset $P$ of non-trivial elements of $L$ such that for every $g \in G$, $g H g^{-1} \cap P$ is not empty. A subset $P$ with this property is called a \textbf{confining subset} for $(H,G)$. 
\end{defin}

Given two subgroups $H,G$ of a group $L$, we denote by $\mathcal{CO}_G^L(H)$ the closure of the $G$-orbit of $H$ under conjugation in the space $\sub(L)$. Hence a subgroup $H$ is confined by $G$ if and only if $\mathcal{CO}_G^L(H)$ does not contain the trivial subgroup $\left\lbrace 1 \right\rbrace $. 

%


%

Although we will be mainly interested in the case $n=1$, the commutator lemma for confined subgroups will be proven in the more general setting of a confined $n$-tuple of subgroups.

\begin{defin}
	Let $n \geq 1$, and $H_1, \ldots, H_n, G$ be subgroups of a group $L$. The $n$-tuple  $(H_1,\ldots,H_n)$ is \textbf{confined} by $G$ if there exists a finite subset $P$ of non-trivial elements of $L$ such that for every $g \in G$, there exists $j$ such that $g H_j g^{-1} \cap P$ is not empty. As above such a subset $P$ is a \textbf{confining subset} for $(H_1,\ldots,H_n,G)$.
\end{defin}

Equivalently, the $n$-tuple $(H_1,\ldots,H_n)$ is confined by $G$ if and only if for the diagonal action of $G$ on the space $\sub(L) \times \cdots \times \sub(L)$, there exists a neighbourhood of the point $\left\{1\right\}^n$ that does not intersect the $G$-orbit of $(H_1,\ldots,H_n)$.

\begin{lem} \label{lem-conf-order2}
	Let $H_1,\ldots, H_n,G$ be two subgroups of a group $L$ such that $(H_1,\ldots, H_n)$ is confined by $G$, and  let $\mathcal{C}$ be the closure of the $G$-orbit of $(H_1,\cdots, H_n)$ in $\sub(L)\times \cdots \times \sub(L)$. Assume that every confining subset $P\subset L$ for $(H_1,\ldots, H_n, G)$ contains at least one element of order $2$. Then there exists $(K_1,\ldots, K_n)\in \mathcal{C}$  such that $K_1,\ldots, K_n$ are all elementary abelian $2$-groups.
\end{lem}

\begin{proof}
We prove the contraposition. For every $\xi=(K_1, \ldots, K_n)\in \mathcal{C}$ we can find $i_{\xi}\in\{1, \ldots, n\}$ and $g_\xi\in K_{i_\xi}$ such that $g_\xi^2\neq 1$.  The sets \[ U(i_\xi, g_\xi) =  \left\lbrace (J_1,\ldots, J_n)\in \mathcal{C}  \, \, : \, \, g_\xi \in J_{i_\xi}\right\rbrace  \] are open and cover $\mathcal{C}$, so we can find a finite subcover $U(i_{\xi_1}, g_{\xi_1}), \ldots, U(i_{\xi_r}, g_{\xi_r})$. This means that $\left\lbrace g_{\xi_1}, \ldots, g_{\xi_r}\right\rbrace $ is a confining subset for $(H_1\ldots, H_n,G)$, and by definition none of these elements have order $2$.
\end{proof}

The following examples show that the notion of confined $n$-tuple appears naturally when dealing with rigid stablisers of group actions. 

\begin{example}
Let $G=\sym_f(\Omega)$ be the group of finitely supported permutations of a set $\Omega$. Consider a cover  $\Omega=\Omega_1\cup \cdots \cup \Omega_n$  of $\Omega$ by non-empty (not necessarily disjoint) subsets, and let $H_i=R_G(\Omega_i)$. Then it is not difficult to check that $(H_1,\cdots, H_n)$ is confined by $G$ (as a confining subset one can take $P$ to be a set of $n+1$ transposition with one common point in their support). However, an individual subgroup $H_i$ is not confined provided that $\Omega_i$ has infinite complement. 
\end{example}

\begin{example}
Let $G$ be any micro-supported group of homeomorphisms of the circle $\mathbb{S}^1$. Consider any cover $\mathbb{S}^1=I_1\cup\cdots\cup I_n$ of $\mathbb{S}^1$ by open intervals, and let $H_i=R_G(I_i)$.  Then it is easy to check as in the previous examples that $(H_1,\ldots, H_n)$ is confined in $G$. However the subgroups $H_i$ need not be: for instance if there exists a sequence $(g_k)$ which accumulates both  endpoints of $I_i$  to the same point, then $(g_kH_ig_k^{-1})_k$ converges to $\{1\}$ in $\sub(G)$. 
\end{example}

\subsection{The case of normal subgroups} \label{subsec-normal}

Before we embark on the proof of Theorem \ref{thm-conf-intro}, we first recall the corresponding statement for normal subgroups. When $N$ is normal in $G$, saying that $P$ is a confining subset for $(N,G)$ is the same as saying that $P$ contains an element $g$ of $N$, and then the singleton $\left\lbrace g \right\rbrace$ is already a confining subset for $(N,G)$. In this setting the notion of displacement configuration that we introduce below is not needed, and Theorem \ref{thm-conf-intro} takes the following form. This statement is well-known, we include a proof for completeness.

\begin{lem}[Classical commutator lemma for normal subgroups] \label{lem-comm-lemma-norm-abs}
Suppose that $G$ acts faithfully on a set $\Omega$, let $N$ be a subgroup of $\Sym(\Omega)$ that is normalized by $G$. Suppose that $\Sigma$ is a subset of  $\Omega$ such that there exists $\sigma \in N$ such that $\Sigma$ and $\sigma(\Sigma)$ are disjoint. Then $N$ contains $\rist_G(\Sigma)'$.
\end{lem} 

\begin{proof}
Take $g_1,g_2 \in \rist_G(\Sigma)$. Since $N$ is normalized by $G$, the element $[g_1,\sigma]$ belongs to $N$, and so does $[[g_1,\sigma],g_2]$. Now since $\Sigma$ and $\sigma(\Sigma)$ are disjoint, $\sigma g_1 \sigma^{-1}$ and $g_2$ commute, and $[[g_1,\sigma],g_2] = [g_1,g_2]$. Therefore $[g_1,g_2]$ belongs to $N$, as desired. 
\end{proof}

\subsection{Notation} \label{subsec-nota}

In the sequel we denote by $\Sym(\Omega)$ the symmetric group on the set $\Omega$. Let $P$ be a finite set of non-trivial elements of $\Sym(\Omega)$ and $\left\lbrace \Omega_\sigma \right\rbrace_{\sigma \in P}$ a collection of subsets of $\Omega$.

Let $H_1,\ldots,H_n \le \Sym(\Omega)$ and $G \le \Sym(\Omega)$ be subgroups of $\Sym(\Omega)$. We denote by $R$ the subgroup of $G$ generated by the rigid stabilizers $\rist_G(\Omega_\sigma)$ for $\sigma \in P$. For every $\sigma \in P$ and $k \leq n$ we let \[ Y_{\sigma,k}=\{\gamma\in R \colon \gamma \sigma \gamma^{-1}\in H_k\}, \, \, \text{and} \, \,  D_{\sigma,k}=\langle \gamma\delta^{-1} \colon \gamma, \delta\in Y_{\sigma,k}\rangle \leq R, \] and for $\gamma, \delta\in Y_{\sigma,k}$ we set \[ a_{\delta, \gamma}=(\delta \sigma^{-1}\delta^{-1})(\gamma \sigma \gamma^{-1}) \, \, \text{and} \, \, A_{\sigma,k} =\langle a_{\delta, \gamma} \colon \gamma, \delta \in Y_{\sigma,k}\rangle \leq H_k . \]

\subsection{A first result} \label{subsec-first-res}

The goal of this paragraph is to prove some preliminary results towards the proof of Theorem \ref{thm-conf-intro}. 

\begin{defin} \label{def-C1-C2}
For a finite set $P$ of non-trivial elements of $\Sym(\Omega)$ and a collection $\left\lbrace \Omega_\sigma \right\rbrace_{\sigma \in P}$ of non-empty subsets of $\Omega$, consider the following properties:
	\begin{enumerate}
		\item[(C1)] for all $\sigma, \rho \in P$, either $\Omega_\sigma = \Omega_\rho$, or $\Omega_\sigma$ and $\Omega_\rho$ are disjoint.
		\item[(C2)] \label{item-one-disj} for all $\sigma \in P$, $\sigma(\Omega_\sigma)$ is disjoint from $\bigcup_{\alpha \in P} \Omega_\alpha$.
	\end{enumerate}
\end{defin}

\begin{defin} \label{def-M-F-P}
Let $P$ and $\left\lbrace \Omega_\sigma \right\rbrace_{\sigma \in P}$ satisfying (C1) and (C2). For $\sigma \in P$, we denote by $M_\sigma \subseteq P$ the set of $\rho \in P$ such that $\sigma(\Omega_\rho)$ is disjoint from $\bigcup_{\alpha \in P} \Omega_\alpha$, and by $F_\sigma \subseteq P$ the set of $\rho \in P$ such that $\sigma$ fixes pointwise $\Omega_\rho$. Clearly $M_\sigma$ and $F_\sigma$ are disjoint, and $\sigma \in M_\sigma$ thanks to condition (C2). 
\end{defin}

	\begin{center}
		\textit{Until the end of \S \ref{subsec-first-res} we assume that $P$ and $\left\lbrace \Omega_\sigma \right\rbrace_{\sigma \in P}$ satisfy conditions (C1) and (C2) of Definition \ref{def-C1-C2}, and we retain the notation $M_\sigma, F_\sigma$ from Definition \ref{def-M-F-P}. We also let $H_1,\ldots,H_n \le \Sym(\Omega)$ and $G \le \Sym(\Omega)$ be subgroups of $\Sym(\Omega)$, and we retain the notation $R, Y_{\sigma,k}, D_{\sigma,k}, A_{\sigma,k}$ from \S \ref{subsec-nota}.}
	\end{center}
	
	\begin{defin} \label{defi-projections}
		We denote by $\mathrm{Stab}(\Omega_\rho)$ the setwise stabilizer of $\Omega_\rho$ in $\Sym(\Omega)$. We also denote by $p_\rho: \mathrm{Stab}(\Omega_\rho) \to \Sym(\Omega_\rho)$ and $\pi_\sigma : \bigcap_{\rho \in M_\sigma} \mathrm{Stab}(\Omega_\rho) \to \Sym(\bigcup_{\rho \in M_\sigma} \Omega_\rho)$ the restriction  maps.
	\end{defin}
	
	 We observe that for any two $\sigma,\rho$ such that $\Omega_\sigma \neq \Omega_\rho$, the rigid stabilizers $\rist_G(\Omega_\sigma)$ and $\rist_G(\Omega_\rho)$ intersect trivially and centralize each other by (C1), so that the subgroup $R$ generated by all the $\rist_G(\Omega_\sigma)$ is the direct product of the $\rist_G(\Omega_\sigma)$, where the product is taken over the set of distinct $\Omega_\sigma$.

	\begin{lem} \label{lem-A-supp}
	For every $\sigma \in P$ and $k \leq n$, the following hold:
		\begin{enumerate}[label=\roman*) ]
			\item \label{item-A-1} The subgroup $A_{\sigma,k}$ is supported in $\bigcup_{\rho \in P \setminus F_\sigma} \Omega_\rho \cup \sigma^{-1}(\Omega_\rho)$.
			\item \label{item-A-2} $A_{\sigma,k}$ preserves $\Omega_\rho$ for all $\rho \in M_\sigma$, and one has $\pi_\sigma(A_{\sigma,k}) = \pi_\sigma(D_{\sigma,k})$. 
		\end{enumerate}
	\end{lem}
	
	\begin{proof}
		Let $\gamma, \delta\in Y_{\sigma,k}$. Recall that $\gamma, \delta$ are supported in $\bigcup_{\rho \in P}\Omega_\rho$ and preserve each $\Omega_\rho$. The element $\sigma^{-1}\delta^{-1} \gamma \sigma$ is therefore supported in $ \bigcup_{\rho \in P} \sigma^{-1}(\Omega_\rho)$ and preserves each $\sigma^{-1}(\Omega_\rho)$. Hence $a_{\delta, \gamma}=\delta (\sigma^{-1}\delta^{-1}\gamma \sigma)\gamma^{-1}$ is supported in $\bigcup_{\rho \in P} \Omega_\rho \cup \sigma^{-1}(\Omega_\rho)$. Since $\sigma$ acts trivially on $\Omega_\rho$ for $\rho \in F_\sigma$, $a_{\delta, \gamma}$ also acts trivially on $\Omega_\rho$. Hence it follows that $a_{\delta, \gamma}$ is actually supported in $\bigcup_{\rho \in P \setminus F_\sigma} \Omega_\rho \cup \sigma^{-1}(\Omega_\rho)$, so \ref{item-A-1} holds.

		 For \ref{item-A-2}, observe that $a_{\delta, \gamma}$ coincides with $\delta \gamma^{-1}$ on $\bigcup_{\rho \in M_\sigma} \Omega_\rho$ because for $\rho \in M_\sigma$ we have that $\sigma(\Omega_\rho)$ is disjoint from $\Omega_\alpha$ for all $\alpha \in P$. Hence $\pi_\sigma(a_{\delta, \gamma}) = \pi_\sigma(\delta \gamma^{-1})$, and the equality $\pi_\sigma(A_{\sigma,k}) = \pi_\sigma(D_{\sigma,k})$ follows since $D_{\sigma,k}$ is generated by the $\gamma\delta^{-1}$ when $\gamma, \delta$ range over $Y_{\sigma,k}$.
	\end{proof}

	\begin{lem} \label{lem-cons-Neum}
 Suppose that $P$ is a confining subset for $(H_1,\ldots,H_n,R)$, and let $r = |P|$. Then there exist $\sigma \in P$ and $k \leq n$ such that $D_{\sigma,k}$ has index at most $nr$ in $R$.
	\end{lem}

\begin{proof}
Saying that $P$ is a confining subset for $(H_1,\ldots,H_n,R)$ is equivalent to saying that the union of the $Y_{\sigma,k}$ for $\sigma \in P$ and $k \leq n$ is equal to $R$. Since for every $g \in Y_{\sigma,k}$ we have the inclusion $Y_{\sigma,k} \subset D_{\sigma,k} g$, it follows that $R$ can be written as a union of cosets of the subgroups $D_{\sigma,k}$ for $\sigma \in P$ and $k \leq n$. Hence according to B.H. Neumann's lemma \cite{Neum54}, there must exist $\sigma \in P$ and $k$ such that $D_{\sigma,k}$ has finite index at most $nr$ in $R$. 
\end{proof}

We isolate the following statement, that will be used in \S \ref{subsec-lower-sc}. Recall that the  the notation $p_\sigma$ was introduced in Definition \ref{defi-projections}.
	
\begin{prop} \label{p-first-step}
Let $H_1,\ldots,H_n \le \Sym(\Omega)$ and $G \le \Sym(\Omega)$ such that $(H_1,\ldots,H_n)$ is confined by $G$. Suppose $P$ is a confining subset for $(H_1,\ldots,H_n,G)$ and $\left\lbrace \Omega_\sigma \right\rbrace_{\sigma \in P}$ satisfies (C1) and (C2). Then there exist $\sigma \in P$ and $k \leq n$ such that $A_{\sigma,k} \leq H_k$ preserves $\Omega_\sigma$, and $p_\sigma(A_{\sigma,k})$ is a subgroup of $\rist_G(\Omega_\sigma)$ of index at most $nr$, where $r = |P|$.
\end{prop}

\begin{proof}
Since $P$ is confining for $(H_1,\ldots,H_n,G)$, it is also confining for $(H_1,\ldots,H_n,R)$. Hence by Lemma \ref{lem-cons-Neum} we can choose $\sigma \in P$ and $k$ such that $D_{\sigma,k}$ has index at most $nr$ in $R$. Thanks to condition (C2) we have $\sigma \in M_\sigma$, so Lemma \ref{lem-A-supp} \ref{item-A-2} implies in particular that $A_{\sigma,k}$ preserves $\Omega_\sigma$ and $p_\sigma(A_{\sigma,k}) = p_\sigma(D_{\sigma,k})$. Since $(\rist_G(\Omega_\sigma) : p_\sigma(D_{\sigma,k})) \leq (R : D_{\sigma,k})$, the statement follows.
\end{proof}

\begin{remark}
Strictly speaking, $\rist_G(\Omega_\sigma)$ should be replaced by $p_\sigma(\rist_G(\Omega_\sigma))$ in the conclusion of the previous proposition. However $p_\sigma$ being injective on restriction to $\rist_G(\Omega_\sigma)$, for ease of notation we freely identify $\rist_G(\Omega_\sigma)$ and $p_\sigma(\rist_G(\Omega_\sigma))$.
\end{remark}

\subsection{The proof of Theorem \ref{thm-conf-intro}} \label{subsec-displace}

The following is a strengthening of the condition Definition \ref{def-C1-C2}.
		
		\begin{defin} \label{def-displace}
			Let $P$ be a finite set of non-trivial elements of $\Sym(\Omega)$. A collection $\left\lbrace \Omega_\sigma \right\rbrace_{\sigma \in P} $ of non-empty subsets of $\Omega$ is a \textbf{displacement configuration} for $P$ if the following hold: 
			\begin{enumerate}
				\item[(C1)] \label{item-disj-eq} for all $\sigma, \rho \in P$, either $\Omega_\sigma = \Omega_\rho$, or $\Omega_\sigma$ and $\Omega_\rho$ are disjoint;
				\item[(C3)] \label{item-one-move} for all $\sigma, \rho \in P$, either $\sigma$ fixes $\Omega_\rho$ pointwise, or $\sigma(\Omega_\rho)$ is disjoint from $\bigcup_{\alpha \in P}\Omega_\alpha$;
				\item[(C4)] \label{item-all-disj} for all $\sigma \in P$, $\sigma(\Omega_\sigma)$ is disjoint from $\bigcup_{\alpha \in P}\Omega_\alpha$ and also from $\bigcup_{\alpha \in P} \sigma^{-1}(\Omega_\alpha)$.
			\end{enumerate}
		\end{defin}
	
\begin{remark}
	We note that condition (C4) implies (C2) from Definition \ref{def-C1-C2}. Note also that (C3) is equivalent to saying that $P = M_\sigma \cup F_\sigma$ for all $\sigma \in P$ (Definition \ref{def-M-F-P}). Combined with (C1), (C3) implies that for $\rho \in M_\sigma$ we have that $\sigma^{-1}(\Omega_\rho)$ is disjoint from $\bigcup_{\alpha \in P} \Omega_\alpha$.
\end{remark}

\begin{center}
	\textit{Until the end of this section we assume that $P$ is a finite set of non-trivial elements of $\Sym(\Omega)$, that $\left\lbrace \Omega_\sigma \right\rbrace_{\sigma \in P} $ is a displacement configuration for $P$, and we retain the notation $M_\sigma, F_\sigma$ from Definition \ref{def-M-F-P}. We also let $H_1,\ldots,H_n$ and $G$ be subgroups of $\Sym(\Omega)$, and we retain the notation $R, Y_{\sigma,k}, D_{\sigma,k}, A_{\sigma,k}$ from \S \ref{subsec-nota}.}
\end{center}

In addition we will also use the following notation. For $\lambda \in Y_{\sigma,k}$, we denote by $B_\lambda$ the conjugate of $A_{\sigma,k}$ by $\lambda \sigma  \lambda^{-1}$: \[ B_\lambda = \lambda \sigma  \lambda^{-1} A_{\sigma,k} \lambda \sigma^{-1} \lambda^{-1}. \] Since $\lambda \in Y_{\sigma,k}$ and $A_{\sigma,k}$ is a subgroup of $H_k$, $B_\lambda$ is also a subgroup of $H_k$.
		
		\begin{lem} \label{lem-B-supp}
		For every $\sigma \in P$, $k \leq n$ and $\lambda \in Y_{\sigma,k}$, the following hold:
		\begin{enumerate}[label=\roman*)] 
			\item \label{lem-B-supp1} $B_\lambda$ is supported in $\bigcup_{\rho \in M_\sigma} \Omega_\rho \cup \sigma(\Omega_\rho)$;
			\item \label{lem-B-supp2} $B_\lambda$ preserves each $\Omega_\rho$ and $\sigma(\Omega_\rho)$, and $\pi_\sigma(B_\lambda) \leq \pi_\sigma(R)$.
		\end{enumerate}
		\end{lem}

\begin{proof}
 For every $\rho \in M_\sigma$, $\sigma^{-1}(\Omega_\rho)$ is disjoint from $\Omega_\alpha$ for all $\alpha \in P$, so the element $\lambda \sigma \lambda^{-1}$ sends $\sigma^{-1}(\Omega_\rho)$ to $\Omega_\rho$ and $\Omega_\rho$ to $\sigma(\Omega_\rho)$. Since $B_\lambda$ is the conjugate of $A_{\sigma,k}$ by $\lambda \sigma  \lambda^{-1}$, statement \ref{lem-B-supp1} follows from Lemma \ref{lem-A-supp} \ref{item-A-1}. Moreover for $\gamma, \delta\in Y_{\sigma,k}$, the element $a_{\delta, \gamma}=\delta (\sigma^{-1}\delta^{-1}\gamma \sigma)\gamma^{-1}$ coincides with $\sigma^{-1}\delta^{-1}\gamma \sigma$ on $\sigma^{-1}(\Omega_\rho)$. So it follows that $A_{\sigma,k}$ preserves $\sigma^{-1}(\Omega_\rho)$. Since $A_{\sigma,k}$ also preserves $\Omega_\rho$, the subgroup  $B_\lambda$ indeed preserves $\Omega_\rho$ and $\sigma(\Omega_\rho)$. So to conclude the proof of \ref{lem-B-supp2} we only have to prove that $\pi_\sigma(B_\lambda) \leq \pi_\sigma(R)$. One has \[ (\lambda \sigma \lambda^{-1}) a_{\delta, \gamma} (\lambda \sigma^{-1} \lambda^{-1}) = \lambda (\sigma \lambda^{-1}  \delta \sigma^{-1} ) \delta^{-1}\gamma (\sigma \gamma^{-1}  \lambda \sigma^{-1}) \lambda^{-1}, \] and the elements $(\sigma \lambda^{-1}  \delta \sigma^{-1} )$ and $(\sigma \gamma^{-1}  \lambda \sigma^{-1})$ both act trivially on $\Omega_\rho$ for $\rho \in M_\sigma$, so it follows from the above equality that $(\lambda \sigma \lambda^{-1}) a_{\delta, \gamma} (\lambda \sigma^{-1} \lambda^{-1})$ coincides with $\lambda \delta^{-1}\gamma \lambda^{-1}$ on $\Omega_\rho$. Hence one has \[ \pi_\sigma(\lambda \sigma_\ell \lambda^{-1} a_{\delta, \gamma} \lambda \sigma_\ell^{-1} \lambda^{-1}) = \pi_\sigma(\lambda \delta^{-1}\gamma \lambda^{-1}),\] and in particular $\pi_\sigma(B_\lambda) \leq \pi_\sigma(R)$ since $\gamma, \delta, \lambda  \in R$.
	\end{proof}

Recall from Lemma \ref{lem-A-supp} that for $\sigma \in P$, the subgroup $A_{\sigma,k}$ is supported in $\bigcup_{\rho \in M_\sigma} \Omega_\rho \cup \sigma^{-1}(\Omega_\rho)$.

\begin{defin}
For $\rho \in M_\sigma$ we denote by $A_{\sigma,k}^\rho$ the subgroup of $A_{\sigma,k}$ consisting of elements supported in $\Omega_\rho \cup \bigcup_{\alpha \in M_\sigma} \sigma^{-1}(\Omega_\alpha)$. Equivalently, $A_{\sigma,k}^\rho$ consists of elements of $A_{\sigma,k}$ acting trivially on $\Omega_\alpha$ for all $\alpha \neq \rho$.
\end{defin} 

\begin{lem} \label{lem-proj-Ai}
For all $\sigma \in P$, one has $\left( \rist_G(\Omega_\rho) : p_\rho(A_{\sigma,k}^\rho)  \right) \leq \left( R : D_{\sigma,k}  \right)$ for every $\rho \in M_\sigma$.
\end{lem}

\begin{proof}
	Indeed one has \[ \left( \rist_G(\Omega_\rho) : p_\rho(A_{\sigma,k}^\rho)   \right) = \left( \rist_G(\Omega_\rho) : \rist_G(\Omega_\rho) \cap D_{\sigma,k}  \right) \leq \left( R : D_{\sigma,k}  \right), \] where the first equality follows from Lemma \ref{lem-A-supp}.
\end{proof}
	

For $k \geq 1$ we will denote by $\mathrm{FC}_{\leqslant k}(G)$ the set of elements of $G$ whose conjugacy class has cardinality at most $k$. Equivalently, $\mathrm{FC}_{\leqslant k}(G)$ is the set of elements of $G$ having a centralizer of index at most $k$. By definition the FC-center of $G$ is $\mathrm{FC}(G) = \bigcup_k \mathrm{FC}_{\leqslant k}(G)$.
	

		\begin{thm} \label{thm-conf-abstract}
			Let $H_1,\ldots,H_n, G \le \Sym(\Omega)$ such that $(H_1,\ldots,H_n)$ is confined by $G$, and let $P$ be a confining subset for $(H_1,\ldots,H_n,G)$, and $r = |P|$. Assume that $\left\lbrace \Omega_\sigma \right\rbrace_{\sigma \in P} $ is a displacement configuration for $P$ such that for all $\sigma \in P$ the group $\rist_G(\Omega_\sigma)$ is non-trivial and satisfies $\mathrm{FC}_{\leqslant nr}(\rist_G(\Omega_\sigma)) = \left\lbrace 1 \right\rbrace $. Then there exist $\rho \in P$ and $k \leq n$ such that $H_k$ contains a non-trivial subgroup $N\le \rist_G(\Omega_\rho)$ whose normalizer in $\rist_G(\Omega_\rho)$ has index at most $nr$. 
		\end{thm}
	
The theorem will follow from the following more precise statement:	
	
\begin{prop} \label{prop-confined-abstract}
Retain the assumptions of Theorem \ref{thm-conf-abstract}. Then for every $\sigma \in P$ and $k \leq n$ such that $D_{\sigma,k}$ has index at most $nr$ in $R$, there exist $\rho \in M_\sigma$ such that $H_k$ contains a non-trivial subgroup $N\le \rist_G(\Omega_\rho)$ that is normalized by $D_{\sigma,k}$.

Moreover if $\rist_G(\Omega_\rho)$ is finitely generated for all $\rho \in M_\sigma$, then one can find a finitely generated subgroup $L$ of $H_k$ that contains $N$.
\end{prop}
	
\begin{proof}
Fix $\sigma$ and $k$ as in the statement. 

\textit{Claim:} There exists $\rho \in M_\sigma$ such that $ H_k \cap \rist_G(\Omega_\rho) \neq 1$. Since $\rist_G(\Omega_\sigma)$ is not trivial and $\mathrm{FC}_{\leqslant nr}(\rist_G(\Omega_\sigma)) = \left\lbrace 1 \right\rbrace$, it follows in particular that $\rist_G(\Omega_\sigma)$ has cardinality strictly larger then $nr$. Since $\left( R : D_{\sigma,k}  \right) \leq nr$, it follows from Lemma \ref{lem-proj-Ai} that $p_\sigma(A_{\sigma,k}^\sigma)$ has index at most $nr$ in $\rist_G(\Omega_\sigma)$, and in particular $p_\sigma(A_{\sigma,k}^\sigma)$ is non-trivial. Choose an element $f\in A_{\sigma,k}^\sigma$ such that $p_\sigma(f) \neq 1$. If $f$ acts trivially on $\sigma^{-1}(\Omega_\rho)$ for every $\rho \in M_\sigma$ then $f$ belongs to $\rist_G(\Omega_\sigma)$ and the claim holds with $\rho = \sigma$. Hence in the sequel we may assume that there exists $\rho \in M_\sigma$ such that $f$ acts non-trivially on $\sigma^{-1}(\Omega_\rho)$.

Fix $\lambda\in Y_{\sigma,k}$, and let $h = (\lambda \sigma \lambda^{-1}) f (\lambda \sigma^{-1} \lambda^{-1}) \in B_\lambda$. Since  $f\in A_{\sigma,k}^\sigma$, the element $h$ is supported in $\bigcup_{\alpha \in M_\sigma} \Omega_\alpha \cup \sigma(\Omega_\sigma)$ and $p_\rho(h) \in p_\rho (\rist_G(\Omega_\rho))$ by Lemma \ref{lem-B-supp}. Moreover by our assumption $p_\rho(h)$ is non-trivial. Since $\mathrm{FC}_{\leqslant nr}(\rist_G(\Omega_\rho)) = \left\lbrace 1 \right\rbrace$, it follows that the index in $\rist_G(\Omega_\rho)$ of the centralizer of  $p_\rho(h)$ is strictly larger than $nr$, and hence by Lemma \ref{lem-proj-Ai} the subgroup  $p_\rho(A_{\sigma,k}^\rho)$ does not centralize $p_\rho(h)$.  
		
		Now by condition (C4) of Definition \ref{def-displace}, we have that $\sigma(\Omega_\sigma)$ is disjoint from $\Omega_\alpha$ and $\sigma^{-1}(\Omega_\alpha)$ for all $\alpha \in P$. Since the subgroup $A_{\sigma,k}^\rho$ is supported in $\Omega_\rho \cup \bigcup_{\alpha \in M_\sigma} \sigma^{-1}(\Omega_\alpha)$, the intersection between the support of $A_{\sigma,k}^\rho$ and the support of $h$ is contained in $\Omega_\rho$. It follows that every element in $[A_{\sigma,k}^\rho,h]$ is supported in $\Omega_\rho$, and that the map $p_\rho$ is injective in restriction to $[A_{\sigma,k}^\rho,h]$. Combined with the previous paragraph, this implies that one can find $a \in A_{\sigma,k}^\rho$ such that $[a,h]$ is non-trivial and $[a,h] \in H_k \cap \rist_G(\Omega_\rho)$. This terminates the proof of the claim.

	To conclude the proof of the proposition, we fix $\rho \in M_\sigma$ and a non-trivial element $h_0 \in H_k \cap \rist_G(\Omega_\rho)$. Let $S$ be a generating set of the group $p_\rho(D_{\sigma,k})$. By Lemma \ref{lem-A-supp} one has $p_\rho(A_{\sigma,k}) = p_\rho(D_{\sigma,k})$, so for every $s \in S$ one can choose $a_s \in A_{\sigma,k}$ such that $p_\rho(a_s) = s$. We denote by $A_S$ the subgroup generated by the elements $a_s$ for $s \in S$, and by  $L$ the subgroup of $H_k$ generated by $h_0$ and $A_S$. We also let $N \leq L$ be the subgroup generated by the elements $a h_0 a^{-1}$, $a \in A_{S}$. Since $h_0 \in \rist_G(\Omega_\rho)$, the subgroup $N$ is contained in $\rist_G(\Omega_\rho)$, and $N$ is not trivial since $h_0$ is non-trivial. Moreover since  \[ \left\lbrace a ha^{-1} \, : \, a \in A_{S} \right\rbrace  = \left\lbrace g h g^{-1} \, : \, g \in D_{\sigma,k} \right\rbrace  \] according to Lemma \ref{lem-A-supp}, the subgroup $N$ is normalized by $D_{\sigma,k}$. Since $N$ is in $H_k$ by construction, it follows that $N$ satisfies all the conclusions.
	
	For the last assertion, if $\rist_G(\Omega_\rho)$ is finitely generated, then so is the finite index subgroup $p_\rho(D_{\sigma,k})$. Hence above it is possible to take for $S$ a finite generating subset of $p_\rho(D_{\sigma,k})$, and it immediately follows that the subgroup $L$ of $H_k$ is also finitely generated.
\end{proof}

\begin{proof}[Proof of Theorem \ref{thm-conf-abstract}]
We choose $\sigma \in P$ and $k \leq n$ such that $D_{\sigma,k}$ has index at most $nr$ in $R$ (these exist by Lemma \ref{lem-cons-Neum}), and apply Proposition \ref{prop-confined-abstract}. Since $N\le \rist_G(\Omega_\rho)$ and $N$ is normalized by $D_{\sigma,k}$, $N$ is also normalized by $p_\rho(D_{\sigma,k})$, which has index at most $nr$ in $\rist_G(\Omega_\rho)$.
\end{proof}

For future reference we highlight the case $n=1$ in Theorem \ref{thm-conf-abstract}:

\begin{thm}[Commutator lemma for confined subgroups] \label{thm-conf-intro} 
	Let $H \leq G$ be a confined subgroup of $G$, and let $P$ be a confining subset for $(H,G)$ or cardinality $r$.  Suppose that $G$ acts faithfully on a set $\Omega$ and that $\left\lbrace \Omega_\sigma \right\rbrace_{\sigma \in P} $ is a collection of subsets of $\Omega$ such that:
	\begin{enumerate}[label=\roman*) ]
		\item  $\left\lbrace \Omega_\sigma \right\rbrace_{\sigma \in P} $  is a displacement configuration for $P$;
		\item $\rist_G(\Omega_\sigma)$ is non-trivial and satisfies $\mathrm{FC}_{\leqslant r}(\rist_G(\Omega_\sigma)) = \left\lbrace 1 \right\rbrace $ for all $\sigma \in P$. 
	\end{enumerate}
	Then there exists $\Omega_\sigma$ such that $H$ contains a non-trivial subgroup $N\le \rist_G(\Omega_\sigma)$ whose normalizer in $\rist_G(\Omega_\sigma)$ has finite index at most $r$. 
\end{thm}

\section{Confined subgroups of groups of homeomorphisms} \label{sec-confined-homeo}

In this section we derive Theorem \ref{thm-doublecomm-conf-homeo-intro} from Theorem \ref{thm-conf-intro}. The following easy lemma shows that the notion of displacement configuration introduced in Definition  \ref{def-displace} is adapted to our current purpose in the case of group actions by homeomorphisms.

\begin{lem}\label{l-partie-finie} 
	Let $P = \left\lbrace \varphi_1,\cdots, \varphi_r\right\rbrace $ be a finite set of homeomorphisms of a Hausdorff space $X$, and let $Y \subset X$ be a non-empty subspace with no isolated points such that for all $i$ the restriction of $\varphi_i^2$ to $Y$ is not trivial. Then there exists a family of open subsets $(U_1,\cdots, U_r)$ of $Y$ that forms a displacement configuration for $P$.
\end{lem}

\begin{proof}

For every $i$, there exist infinitely many points $y\in Y$ such that $\varphi^{-1}_i(y), y$ and $\varphi_i(y)$ are pairwise distinct, thanks the assumption that $\varphi_i^{2}$ is not the identity on $Y$ and that $Y$ has no isolated point. Since the space is Hausdorff, this implies that we can find non-empty open subsets $V_1,\cdots, V_r \subset Y$ such that the $3r$ sets $\varphi_i^{\epsilon}(V_i)$ are pairwise disjoint for $\epsilon =-1, 0, 1$ and $i=1,\ldots, r$. Note that this condition is still satisfied if we replace the sets $V_i$ by any non-empty open subsets $V_i'\subset V_i$. We will choose the sets $U_i$ inside $V_i$. For any such choice, condition (C1) is automatically satisfied as well as the part of conditions (C3) and (C4) stating that $\varphi_i(U_i)$ is disjoint from $\cup_i U_i$.
For every  $i\neq  j$, we can make sure that $\varphi_i(V_j)$ is disjoint from $\cup_i V_i$ by shrinking the sets $V_i$ unless $\varphi_i$ is the identity on $V_j$, and this condition is again preserved  still by further shrinking.  Thus we can ensure that (C3) is satisfied in finitely many steps.  Since moreover $\varphi_i$ is not the identity on $V_i$ we can argue similarly to ensure condition (C4).  \qedhere

\end{proof}

For a proof of the following, see the proof of Theorem 9.17 from \cite{Grig-dynamics-trees} or the proof of Theorem 1.2 in \cite{Zheng-irs}. 

\begin{lem} \label{lem-rist-FC}
	Let $G$ be a micro-supported group of homeomorphisms of a Hausdorff space $X$. For every open subset $U \subset X$, the rigid stabilizer $\rist_G(U)$ has trivial FC-center.
\end{lem}

\begin{lem} \label{lem-conf-2group}
	Let $G$ be a micro-supported group of homeomorphisms of a Hausdorff space $X$, and let $H_1,\ldots,H_n \leq \homeo(X)$ such that $(H_1,\ldots,H_n)$ is confined by $G$. Then there is $\ell$ such that $H_\ell$ is not an elementary abelian $2$-group.
\end{lem}

\begin{proof}
	Let $P=\{\varphi_1,\ldots, \varphi_r\}$ be a confining subset for $(H_1,\ldots,H_n,G)$. Since $\varphi_1,\ldots, \varphi_r$ are non-trivial, there are non-empty open subsets $U_1,\ldots, U_r$ such that $U_1,\ldots, U_r$, $\varphi_1(U_1),\ldots, \varphi_r(U_r)$ are all disjoint. Since rigid stabilizers are not abelian (Lemma \ref{lem-rist-FC}), in particular they are not of exponent $2$, so for each $i$ we may find disjoint open subsets $U_{i,1},\ldots, U_{i,nr+1} \subset U_i$ and $\gamma_{i,j} \in \rist_G(U_{i,j})$ such that $\gamma_{i,j}$ is not of order $2$. Let $\gamma_j = \gamma_{1,j} \cdots \gamma_{r,j}$. Since $P$ is confining for $(H_1,\ldots,H_n,G)$, by the pigeon-hole principle, we can find $j \neq k$ and $i,\ell$ such that $a = \gamma_j \varphi_i \gamma_j^{-1}$ and $b = \gamma_k \varphi_i \gamma_k^{-1}$ both belong to $H_\ell$. Then $h = a^{-1}b$ is an element of $H_\ell$ that preserves $U_{i,j}$ and coincides with $\gamma_{i,j}$ on $U_{i,j}$. It follows that $h$ is not of order $2$, and $H_\ell$ is not an elementary abelian $2$-group.
\end{proof}

Recall the classical commutator lemma for normal subgroups of groups of homeomorphisms, that follows from Lemma \ref{lem-comm-lemma-norm-abs}.

\begin{lem} \label{lem-double}
Let $G$ be a group of homeomorphisms of a Hausdorff space $X$, and let $N \leq G$ be a non-trivial normal subgroup. Then there exists a non-empty open subset $U\subset X$ such that $N$ contains $\rist_G(U)'$. More precisely, for every open subset $U$ such that there exists $\sigma \in N$ such that $\sigma(U)$ and $U$ are disjoint, $N$ contains $\rist_G(U)'$.
\end{lem}



Theorem \ref{thm-doublecomm-conf-homeo-intro} follows from the following theorem.

\begin{thm} \label{thm-doublecomm-conf-homeo}
	Let $G$ be a group of homeomorphisms of a Hausdorff space $X$ and let $H_1,\ldots,H_n \leq \homeo(X)$ such that $(H_1,\ldots,H_n)$ is confined by $G$. Then there exist $k$ and a non-empty open subset $U\subset X$ such that $H_k$ contains $\rist_G(U)'$. 
\end{thm}

\begin{proof}
	We may assume that the action of $G$ on $X$  is micro-supported, because otherwise the conclusion is trivially satisfied. Observe that if $(K_1,\ldots, K_n)$ belongs to the  $G$-orbit closure of $(H_1, \ldots, H_n)$, then $(K_1, \ldots, K_n)$ is still confined by $G$. Therefore according to Lemma \ref{lem-conf-2group} and Lemma \ref{lem-conf-order2}, there must exist a finite subset $P$ of $\homeo(X)$ that is confining for $(H_1,\ldots,H_n,G)$ such that $P$ contains no element of order $2$. Lemma \ref{l-partie-finie} asserts that we may find a displacement configuration $(U_1,\cdots, U_r)$ for $P$ consisting of non-empty open subsets of $X$. The fact that the action of $G$ on $X$ is micro-supported implies that all rigid stabilizers $\rist_G(U_i)$ have trivial FC-center (Lemma \ref{lem-rist-FC}), so it follows that all the assumptions of Theorem \ref{thm-conf-abstract} are satisfied. The conclusion of that theorem ensures that there is $k$ and $U_i = U$ such that $H_k$ contains a non-trivial subgroup $N$ of $\rist_G(U)$ that is normalized by a finite index normal subgroup $L$ of $\rist_G(U)$. Since $N$ is non-trivial and normalized by $L$, by Lemma \ref{lem-double} there exists an open subset $V$ of $U$ such that $N$ contains $\rist_L(V)'$. Now $\rist_L(V)$ is normal in $\rist_G(V)$, and hence so is $\rist_L(V)'$, so that by applying Lemma \ref{lem-double} again we find a non-empty open subset $W$ such that $\rist_L(V)'$ contains $\rist_G(W)'$. Hence $\rist_G(W)' \leq H_k$, and we are done.
\end{proof}




\section{Groups acting on rooted trees I: structure of non-free minimal actions} \label{sec-trees-min-actions}

\subsection{Preliminaries on group actions on rooted trees}

In the sequel $T$ is a locally finite rooted tree, with root $o$. We denote by $\partial T$ is the visual boundary of $T$. We let $\aut(T)$ be the group of automorphisms of $T$ that fix the root. Note that every $G \le \aut(T)$ acts on $\partial T$ by homeomorphisms.  The following classical fact provides a characterisation of  actions that arise in this way. Recall that if a group $G$ acts on a compact space $X$, the $G$-action is \textbf{profinite} if it is the inverse limit of finite $G$-actions. When this holds we will also say that the $G$-space $X$ is profinite. 

\begin{prop} \label{p-odometers}
	Let $G$ be a group acting by homeomorphisms on a totally disconnected compact metrizable space $X$. The following are equivalent:
	\begin{enumerate}[label=\roman*)]
		\item The action of $G$ on $X$ is profinite.
		\item Continuous equivariant maps from  $X$ to finite $G$-spaces separate points in $X$.
		\item \label{i-clopen-finite} Every clopen subset of $X$ has a finite $G$-orbit.
		\item There exists a $G$-action by automorphisms on a locally finite rooted tree $T$ such that the action of $G$ on $\partial  T$ is topologically conjugate to its action on $X$.
		\item \label{i-ultrametric} The $G$-action on $X$ preserves a compatible distance on $X$.
\end{enumerate}
Moreover if this holds, then:

\begin{itemize}
\item[(a)] the orbit closure of every $x\in X$ is minimal;
\item[(b)]  every $G$-action on a compact space $Y$ which is a factor of the action on $X$ is again profinite. 
\end{itemize}
\end{prop}

\begin{proof}
Some implications are immediate. For the others, see  \cite[Proposition 6.4]{Grig-Nek-Sush}. Statement (a) follows from \ref{i-ultrametric}, and  (b) follows from \ref{i-clopen-finite}.
\end{proof}


The set of vertices of a locally finite rooted tree $T$ at distance $n$ from the root form the $n$-th \emph{level} of $T$,  denoted  by $\L(n)$. A subgroup $G\le \aut(T)$ is said to be \textbf{level-transitive} if its action on $\L(n)$ is transitive for every $n$; this is equivalent to the minimality of the action of $G$ on $\partial T$. Note that this is possible only if the tree $T$ is \textbf{spherically homogeneous}, that is, any two vertices at the same level have the same degree.

 Given vertices $v, w\in T$, we say that $w$ is \textbf{below} $v$  if the unique geodesic from the root to $w$ passes through $v$. We denote by $T_v$ the subtree of $T$  of vertices below $v$, and by $\partial T_v$ the corresponding clopen subset in the boundary $\partial T$. 

Given a subgroup $G \leq \aut(T)$ and a vertex $v$, we denote by $\st_G(v)$ the stabilizer of $v$ in $G$, and by $\rist_G(v)$ the rigid stabilizer of $\partial T_v$ in $\partial T$, i.e.\ the elements of $G$ which fix pointwise the complement of $T_v$.  For $n \geq 1$, we denote by $\st_G(n)$ the $n$-th level stabilizer in $G$, i.e.\ the intersection of $\st_G(v)$ for $v \in \L(n)$. We also denote by $\rist_G(n) \simeq \prod_{v\in \L(n)} \rist_G(v)$ the subgroup generated by $\rist_G(v)$ when $v$ ranges over $\L(n)$.

\begin{defin} 
A group $G$ is a \textbf{weakly branch group} if $G$ admits a faithful action on a locally finite rooted tree $T$ such that the associated action on $\partial T$ is minimal and micro-supported. This last condition is equivalent to saying that $\rist_G(v)$ is non-trivial for every vertex $v\in T$.
\end{defin}

\begin{defin} 
	A group $G$ is a \textbf{branch group} if $G$ admits a faithful action on a locally finite rooted tree $T$ such that the associated action on $\partial T$ is minimal and micro-supported, and moreover for every $n \ge 1$ the $n$-th level rigid stabiliser $\rist_G(n)$ has finite index in $G$.
\end{defin}

In the sequel for simplicity when we say that a subgroup $G$ of $ \aut(T)$ is a weakly branch group or a branch group, we implicitly mean that the $G$-action on $T$ satisfies the defining conditions above. 

\begin{lem} \label{lem-wb-xi1xi2}
Let $G\le \aut(T)$ be a weakly branch group. Then for every $\xi_1 \neq \xi_2 \in \partial T$, there exists $n$ such that the subgroup generated by $G_{\xi_1}^0$ and $G_{\xi_2}^0$ contains $\rist_G(n)$.
\end{lem}

\begin{proof}
Since $\xi_1 \neq \xi_2 $ there exists $n$ such that $\xi_1$ and $\xi_2 $ are separated by vertices of level $n$, and it follows that for every $v \in \L(n)$ we have $\rist_G(v) \leq G_{\xi_1}^0$ or $\rist_G(v) \leq G_{\xi_2}^0$.
\end{proof}

For groups acting on rooted trees, the commutator lemma for normal subgroup gives the following fundamental fact, used by Grigorchuk in \cite{Gri-just-infinite}. 

\begin{lem}[Grigorchuk] \label{l-Gri-branch}
Let $G\le \aut(T)$ be a  level-transitive subgroup of automorphisms of a rooted tree. Then every non-trivial normal subgroup of $G$ contains $\rist_G(n)'$ for some $n\ge 1$. In particular, if $G$ is a branch group,  every proper quotient of $G$ is virtually abelian. 
\end{lem}

We  will also use the following result due to Francoeur. 

\begin{thm}[Francoeur] \label{t-Francoeur}
If $G\le \aut(T)$ is  a finitely generated branch group, all normal subgroups of $G$ are finitely generated. As a consequence, for every $v\in T$ the group $\rist_G(v)'$ is finitely generated. 
\end{thm}

\begin{proof} The fact that the normal subgroups are finitely generated is proven in \cite{Fra-append}. In particular for every level $n$ the group $ \rist_G(n)'=\prod_{v\in \L(n)} \rist_G(v)'$ is finitely generated, and it follows that  all the groups $\rist_G(v)'$ must be finitely generated as well. \qedhere
\end{proof}

\subsection{A characterization of confined subgroups in weakly branch groups} 

Theorem \ref{thm-doublecomm-conf-homeo} implies that weakly branch groups enjoy the following simple characterization of confined subgroups.

\begin{cor} \label{cor-confined-wb}
Let $G\le \aut(T)$ be a weakly branch group, and let $H$ be a subgroup of $G$. Then $H$ is confined if and only if there exists a vertex $v$ such that $\rist_G(v)' \leq H$.
\end{cor}

\begin{proof}
Being characteristic in $\rist_G(v)$, the subgroup  $\rist_G(v)'$ is non-trivial (by Lemma \ref{lem-rist-FC}) and normal in $\st_G(n)$, where $n$ is the level of $v$. Hence $\rist_G(v)'$ has only finitely many conjugates. Since $\rist_G(v)'$ is non-trivial, it is then confined, and hence so is every subgroup $H$ of $G$ containing it. The converse implication is provided by Theorem \ref{thm-doublecomm-conf-homeo}.
\end{proof}

\subsection{Two constructions of URSs in weakly branch groups} \label{subsec-2-constructions}

In this paragraph we explain the construction of two families of examples of URSs in weakly branch groups. The first arises by looking at the action on the space of closed subsets of the boundary of the tree (which is endowed with the Chabauty topology).

We will use the following observation.

\begin{lem} \label{lem-2^X-odom}
	If the action of a group $G$ on a compact space $X$ is profinite, then the induced action of $G$ on $\F(X)$ is also profinite.  
\end{lem}

\begin{proof}
	Recall from Proposition \ref{p-odometers} that profinite is equivalent to saying that every clopen subset has a finite $G$-orbit. The space $X$ being totally disconnected, the family of subsets of the form
	\[ \left\{C \in \mathcal{F}(X)\, : \, C \cap K = \emptyset; \, C \cap U_i \neq \emptyset \, \, \text{for all $i$} \, \right\} \] where $K, U_1,\ldots,U_n \subset X$ are clopen,  forms a basis of clopen subsets for the topology on $\F(X)$. Since the $G$-action on $X$ is profinite, each clopen subset in $X$ has a finite $G$-orbit. It follows that each subset of $\F(X)$ as above also has a finite $G$-orbit, and that the $G$-action on $\F(X)$ is profinite.
\end{proof}

%
%

Let $G\le \aut(T)$. Given $C \in \F(\partial T)$, we denote by $\fix_G(C)=\{g\in G \colon g(x)=x \quad \forall x\in C\}$ its pointwise fixator.  It is not difficult to check that the map
\[ \operatorname{Fix}_G  \colon  \F(\partial T) \to \sub(G)\]
is upper semi-continuous. This can be used to define a vast family of URSs of $G$ as follows.

\begin{prop} \label{p-wb-fix}
Let $G\le \aut(T)$ be a weakly branch group.
\begin{enumerate}[label=\roman*)]
\item \label{i-urs-wb-existence} For every $C\in \F(\partial T)$, the closure of the conjugacy class of $\fix_G(C)$ in $\sub(G)$ contains a unique uniformly recurrent subgroup of $G$, denoted $\H_C$. If moreover $G$ is countable then $\fix_G(C)\in \H_C$ for $C$ in a dense $G_\delta$-subset of $\F(\partial T)$.

\item  \label{i-urs-wb-distinct} Given $C, D\in \F(\partial T)$, we have $\H_C=\H_D$ if and only if $C$ and $D$ have the same orbit closure in $\F(\partial T)$. 
\end{enumerate}
\end{prop}

\begin{proof}
Given $C\in \F(\partial T)$, the set $X = \overline{G \cdot C} \subset \F(\partial T)$ is a minimal $G$-space by Lemma \ref{lem-2^X-odom} and Proposition \ref{p-odometers}. Thus it follows from the upper semi-continuity of $\fix_G$ and  \cite[Lemma I.1]{Ausl-Glasn} that the closure of the set $\left\lbrace \fix_G(D) \, : \, D \in X\right\rbrace $ contains a unique URS. If moreover, $G$ is countable, the map $\fix_G$ is continuous on a dense $G_\delta$-subset of $\F(\partial T)$ by semi-continuity, and $\fix_G(C)\in \H_C$ for every continuity point $C$ (see the argument in \cite[Theorem 2.3]{Glasner-compress}).

Let us now check \ref{i-urs-wb-distinct}. By construction it is clear that if $C$ and $D$ have the same orbit closure then $\H_C=\H_D$. For a closed subset $C\in \F(\partial T)$ and $n \geq 1$, let $A_n(C)\subset \L(n)$ be the set of vertices $v\in \L(n)$ such that $\partial T_v \cap C=\varnothing$. Note that when $C$ is fixed,  the sets $U_n=\{D\in \F(\partial T) \colon A_n(D)=A_n(C)\}$ form a basis of clopen neighbourhoods of $C$ in $\F(\partial T)$. It follows that  $C\in \overline{G \cdot D}$ if and only if for every $n$ the sets $A_n(C)$ and $A_n(D)$ belong to the same $G$-orbit in the set $2^{\L(n)}$ of subsets of the $n$-th level. Thus it is enough to show that  the sequence of orbits of the sets $A_n(C)$ can be recovered from $\H_C$.
To this end, for every subgroup $H\in \sub(G)$ we can consider the sequence of sets $\tilde{A}_n(H)=\{v\in \L(n) \colon \rist_G(v) \le H\}\subset \L(n)$.

Let $H \in \H_C$ and take a net $(g_i)$ such that $H_i := g_i \fix_G(C) g_i^{-1} = \fix_G(g_i(C))$ converges to $H$. Fix an integer $n$. Upon passing to a subnet, we may assume that $g_i(A_n(C))$ is constant, denote it $A_n$. For every $v \in A_n$ we have $\rist_G(v) \leq H_i$, and hence $\rist_G(v) \leq H$ since $H_i$ converges to $H$. So  $A_n \subseteq \tilde{A}_n(H)$. Conversely if $v \in \tilde{A}_n(H)$ then $H$ does not fix any point in $\partial T_v$, and hence by compactness $H$ admits a finitely generated subgroup $\Gamma$ with the same property. Eventually we have $\Gamma \leq H_i$, and hence $v \in A_n$. Therefore $\tilde{A}_n(H)$ is equal to $A_n$, and hence belongs to the $G$-orbit of $A_n(C)$, as desired.
\end{proof}

\begin{cor}
If $G\le \aut(T)$ is a weakly branch group, then $G$ admits uncountably many distinct URSs. 
\end{cor}

\begin{proof} By Proposition \ref{p-wb-fix}, it is enough to show that there exist uncountably many closed subsets $C\subset \F(\partial T)$ with pairwise different orbit-closures. For example, here is one way to see this. Let $d_n$ be the sequence of degrees of vertices of $T$.  Fix a  sequence $(\ell_n)$ such that $ 1\le \ell _n \le d_n$, and choose a  spherically homogeneous rooted subtree $T'\subset T$ with sequence of degrees $(\ell_n)$, and consider the set $C=\partial T'$.   Then the sets $A_n(C)$ from the proof of Proposition \ref{p-wb-fix} consist exactly of the level $n$ vertices of $T'$ thus satisfy $|A_n(C)|=\ell_1\cdots \ell_n$. By the same argument in the proposition the sequence of cardinalities $|A_n(C)|$ is an invariant of the orbit-closure of $C$. Thus if we let the sequence $(\ell_n)$ vary, we easily construct uncountably many sets with distinct orbit-closures.  \qedhere

\end{proof}

We now explain a natural variant of the previous construction, which also plays a role in the sequel (Remark \ref{rmk-sandwich}).

\begin{notation}
We let $\P(T)$ be the power set of the set of vertices of $T$, endowed with the natural product topology.
\end{notation}

\begin{defin}\label{d-antichain}
Two vertices $v, w\in T$ are \textbf{independent} if $T_v \cap T_w= \varnothing$. We let $\Pi(T) \subset \P(T)$ be the set of $\V \in \P(T)$ such that $\V$ consists of pairwise independent vertices. We  endow $\Pi(T)$ with the topology induced by $\P(T)$, which makes it a compact space. 
\end{defin}

\begin{defin}
We define a map $\pi_\perp : \Pi(T) \to \F(\partial T)$ by  $\pi_\perp(\V) = \partial T \setminus \sqcup_{v\in \V} \partial T_v$.
\end{defin}

 The map $\pi_\perp : \Pi(T) \to \F(\partial T)$ is continuous and surjective, and for $C \in \F(\partial T)$, the preimage $\pi_\perp^{-1}(C)$ has a natural identification with the partitions of the complement of $C$ into cylinder sets. 
 

Let $G$ be a subgroup of $\aut(T)$. To every $\V\in \P_\perp(T)$, we can naturally associate a subgroup of $G$, namely:
 \[\rist_G(\mathcal{V}):= \langle \rist_G(v) \colon v\in \mathcal{V} \rangle \simeq \bigoplus_{v\in \V} \rist_G(v),\]
 and by definition this subgroup lies in $\fix_G(\pi_\perp(\V))$.
 
\begin{lem} \label{lem-V-sub-lsc}
The map $\P_\perp(T) \to \sub(G)$, $\V \mapsto \rist_G(\mathcal{V})$, is lower semi-continuous. 
\end{lem}

\begin{proof}
Suppose that $(\V_n)$ converges to $\V$ and that the sequence of subgroups $(\rist_G(\V_n))$ converges to a subgroup $H$ of $G$. One has to check that $\rist_G(\V) \leq H$. Let $v \in \V$ and $g \in \rist_G(v)$. For $n$ large enough, we have $v \in \V_n$ since $(\V_n)$ converges to $\V$, and hence $g \in \rist_G(\V_n)$. Therefore it follows that $g \in H$ since  $(\rist_G(\V_n))$ converges to $H$, and hence $\rist_G(v) \leq H$. Since $v$ was arbitrary, this shows $\rist_G(\V) \leq H$.
\end{proof}

This map can be used to give another construction of URSs in weakly branch groups, in a way which is completely analogous to Proposition \ref{p-wb-fix}.

\begin{prop} \label{p-wb-urs-rist} Let $G\le \aut(T)$ be a weakly branch group. For every $\V \in \Pi(T)$, the closure of the conjugacy class of $\rist_G(\V)$ in $\sub(G)$ contains a unique URS of $G$.  Moreover if $G$ is countable, the subgroup $\rist_G(\V)$ belongs to a URS for any point $\V$ which is a continuity point of this map. The same is true if $\rist_G(\V)$ is replaced by its derived subgroup $\rist_G(\V)'$. 
\end{prop}

\begin{proof}
	The first observation is that the $G$-action on $\Pi(T)$ is profinite. Indeed since the action of $G$ on $T$ has only finite orbits, it follows that the action of $G$ on $\P(T)$ is profinite, and thus so is the action of $G$ on the closed invariant subset $\Pi(T)$. Since in addition the map $\V\mapsto \rist_G(\V)$ is lower semi-continuous by Lemma \ref{lem-V-sub-lsc}, the statement follows from \cite[ Lemma I.1]{Ausl-Glasn} (and from \cite[Theorem 2.3]{Glasner-compress} in the case of a countable $G$). The argument for $\rist_G(\V)'$ is the same. 
\end{proof}

\subsection{The structure theorem on URSs and its consequences}

\begin{defin}
Given a subgroup $H \leq \aut(T)$, we denote by $\fix(H)$ the set of fixed points of $H$ in $\partial T$. If $G\leq \aut(T)$ and if $\H$ is a URS of $G$, we will denote by \[F_\H = \left\lbrace \fix(H) \, : \, H \in \H \right\rbrace. \]
\end{defin}

\begin{remark}
In order to stick to the notation introduced in Section  \ref{s-preliminaries}, we should rather write $\fix_{\partial T}(H)$ instead of $\fix(H)$. But since there is no possible confusion here we use $\fix(H)$ in order to simplify the notation. 
\end{remark}

We now state our main structure theorem for URSs of weakly branch groups. It is a general fact that when $\H$ is a URS of a group $G$ and $G$ acts on a compact space $Y$, the map $\H\to \F(Y)$, $H\mapsto \fix_Y(H)$, is upper semi-continuous (Lemma \ref{lem-compos-fix-lsc-usc}). The first assertion of the theorem is that in the case where the action on $Y$ is profinite, this map is actually continuous (this statement does not require $G$ to be weakly branch, see Proposition \ref{p-fix-usc}). Now for $G$ weakly branch, the second assertion of the theorem says that a lot of information can be recovered on $H$ from the knowledge of $\fix(H)$: one can find a partition of the complement of $\fix(H)$ into cylinders sets, i.e.\ an element $\V_H \in \Pi(T)$ such that $\pi_\perp(\V_H) = \fix(H)$, such that $H$ contains the subgroup $\rist_G(v)'$ for all $v \in \V_H$. Moreover it is possible to find such a $\V_H \in \Pi(T)$ that varies continuously with $H$, and in such a way that $\V_H$ depends only on $\fix(H)$ (and not on the subgroup $H$ itself). This is summarized as follows:

 \begin{thm} \label{t-wb-urs}
Let $G\leq \aut(T)$  be a weakly branch group, and $\H$ be a URS of $G$. Then the following hold:
\begin{enumerate}[label=\roman*)]
	\item The map $\H\to \F(\partial T)$, $H\mapsto \fix(H)$, is continuous.
	\item There exists a continuous $G$-equivariant map $\sigma: F_\H\to \Pi(T)$ such that: \begin{enumerate}[label=\alph*)] \item $\pi_\perp \circ \sigma = id$, i.e.\ $\sigma$ is a section of the projection $\pi_\perp$; \item for every $H\in \H$ and $v\in \sigma(\fix(H))$ we have $ \rist_G(v)' \le  H$.
		\end{enumerate}
\end{enumerate}
\end{thm}

\begin{remark} \label{rmk-sandwich}
	Note that using notation from \S \ref{subsec-2-constructions}, saying that $ \rist_G(v)' \le  H$ for all $v\in \sigma(\fix(H))$ is equivalent to saying that $\rist_G(\sigma(\fix(H)))'  \le  H$. Since the inclusion $H \le \fix_G(\fix(H))$ also holds by definition, in fact we have the double inclusion
\[ \rist_G(\sigma(\fix(H)))'  \le  H \le \fix_G(\fix(H)).\]

Thus, although the knowledge of $\fix(H)$ is not enough to recover the subgroup $H$ completely,  Theorem \ref{t-wb-urs} implies in particular that every URS of a weakly branch group  is sandwiched between two URSs as in Propositions \ref{p-wb-fix} and \ref{p-wb-urs-rist}. 
\end{remark}

The proof of Theorem \ref{t-wb-urs} is postponed to \S \ref{subsec-lower-sc}.  We discuss some of its consequences.

\begin{cor} \label{cor-urs-parabolic}
Let $G \leq \aut(T)$ be a weakly branch group, and let $\H$ be a URS   of $G$. Then exactly one of the following happens:
\begin{enumerate}[label=\roman*)]
	 \item For every $H \in \H$, the subgroup $H$ fixes a point in $\partial T$;
	 \item there exists a level $n$ such that $\rist_G(n)'\le H$ for every $H \in \H$.
\end{enumerate}
\end{cor}

\begin{proof}
If there exists $H\in \H$ such that $\fix(H)= \varnothing $, then by minimality and continuity of the map $\operatorname{Fix}$ we have $F_\H = \left\lbrace \varnothing \right\rbrace$, and by Theorem \ref{t-wb-urs} there exists $\V\in \Pi(T)$ such that $\partial T=\sqcup_{v\in \V} \partial T_v$ and $ \rist_G(v)' \le  H$ for every $H \in \H$ and $v\in \V$ . By compactness $\V$ is necessarily finite, and it follows that if we choose $n$ such that every $v \in \V$ is at level at most $n$, then we have $\rist_G(n)'\le H$. Finally these two situations are mutually disjoint because $\rist_G(n)'$ does not fix any point in $\partial T$, as follows for instance from Lemma \ref{l-wb-rigid-move-derived}. \qedhere
\end{proof}

\begin{remark}
In terms of the partial order $\preceq$ on the set of URSs of $G$, the first condition of Corollary \ref{cor-urs-parabolic} is equivalent to $\H \preceq \mathcal{S}_{G}(\partial T)$.
\end{remark}

Another consequence of Theorem \ref{t-wb-urs} is that for every URS $\H$ of $G$, the action of $G$ on $\H$ factors onto the set $F_\H\subset \F(\partial T)$, which is a profinite $G$-space (Lemma \ref{lem-2^X-odom}). The following corollary describes the situations where the space $F_\H$ is finite.

\begin{cor} \label{c-wb-urs-infinite}
Let $G \leq \aut(T)$ be a weakly branch group, and $\H$ a URS of $G$ such that $F_\H$ is finite. Then there exists a level $n$ such that the normal subgroup $N = \rist_G(n)'$ acts trivially on $\H$, and the map $\H \to \sub(G/N)$, $H \mapsto NH/N$, is continuous and finite-to-one. In particular if $G$ is a branch group, then $\H$ is a finite URS. 
\end{cor}

\begin{proof}
If $C\in \F(\partial T)$ has a finite $G$-orbit, then the stabilizer of $C$ in $G$ is a finite index subgroup of $G$, and Lemma \ref{l-minimal-fi} implies that $C$ is clopen. So if the set $F_\H$ is finite, then it consists of  clopen subsets of $\partial T$. It follows that for every $C\in F_\H$, the set $\sigma(C) \in \Pi(T)$ consists of finitely many vertices. Thus we can find a level $n$ such that for every $v\in \L(n)$ and every $C\in F_\H$, we have either $\partial T_v\subset C$, or that $\partial T_v\subset \partial T_w$ for some $w\in \sigma(C)$. In particular for every $H\in \H$ we can find a partition $\L(n)= \P_H^+ \cup \P_H^-$, where $\P_H^+$ consists of vertices $v$ such that $\partial T_v\subset \fix(H)$, and $\P_H^-$ of vertices $v$ such that $\partial T_v \subset \partial T_w$ for some $w\in \sigma(\fix(H))$. 
 Let $N=\rist_G(n)'$. The group $N$ acts trivially on $\H$ because for every $H \in \H$ and $v\in \L(n)$, either $\rist_G(v)'\le H$ (if $v\in \P_H^-$) or $\rist_G(v)'$ centralizes $H$ (if $v\in \P_H^+$).

Now given $H \in \H$, it follows from the above description that $NH = \prod_{v \in \P_H^+} \rist_G(v)' \times H$. From this we deduce that $H \mapsto NH$ is continuous and is injective in restriction to each fiber of the map that associate to every $H$ the corresponding partitions $\P_H^+\sqcup \P_H^-$. Since there are finitely many such fibers, this map is finite-to-one. Since the natural map from $\sub(G/N)$ to $\sub(G)$ is a $G$-equivariant homeomorphism onto its image, the statement follows. 

If $G$ is a branch group, the quotient $G/N$ is virtually abelian by Lemma \ref{l-Gri-branch}. Hence the image of $\H$ in $\sub(G/N)$ is finite, and therefore $\H$ is also finite.
\end{proof}

\begin{remark}
When the URS $\mathcal{H}$ is finite, the conclusion of Theorem \ref{t-wb-urs} is equivalent to a result of  Garrido and Wilson \cite[Theorem 1.2]{Gar-Wil} on the structure of subgroups of branch groups with a finite conjugacy class.
\end{remark}

\subsection{The main structure theorem for non-topologically free minimal actions} \label{subsec-lower-sc}

In this section, we prove  Theorem \ref{t-wb-urs}. In fact we will prove the following more general theorem, which is the core of this section:

\begin{thm} \label{t-wb-structure-lsc}
Let $G\leq \aut(T)$ be a weakly branch group, $X$ a minimal compact $G$-space, and $\Phi \colon X \to \sub(G)$ a lower semi-continuous $G$-equivariant map. Then the following hold:
\begin{enumerate}[label=\roman*)]
	\item The map $\Psi\colon X \to \F(\partial T)$, $x \mapsto \fix(\Phi(x))$, is continuous.
	\item There exists a continuous $G$-map $\sigma: F_X \to \Pi(T)$, where $F_X$ is the image of $\Psi$, such that: \begin{enumerate}[label=\alph*)] \item $\pi_\perp \circ \sigma = \operatorname{id}$, i.e.\ $\sigma$ is a section of the projection $\pi_\perp$; \item for every $x \in X$ and $v\in \sigma(\fix(\Phi(x)))$ we have $ \rist_G(v)' \le  \Phi(x)$.
	\end{enumerate}
\end{enumerate}
\end{thm}

\begin{remark}
This theorem implies Theorem \ref{t-wb-urs} by taking $X=\H$ a URS of $G$ and $\Phi\colon \H \hookrightarrow \sub(G)$ the inclusion (which is thus continuous). The main advantage of stating this result in this more general setting is that it  also leads to a rigidity result for non-topologically free minimal $G$-actions on compact spaces, by considering the map $\Phi(x)=G_x^0$. In particular, it implies that every non-topologically free minimal action of $G$ must factor onto a non-trivial closed $G$-invariant subspace of $\F(\partial T)$ (see \S \ref{subsec-factor-profinite}).
\end{remark}



We will  need the following elementary lemmas.

 \begin{lem} \label{l-wb-rigid-move-fi}
Let $G\le \aut(T)$ be a weakly branch group. Let $v \in T$, and $k \ge 1$. Then there exists $m=m(v, k)$ such that  every orbit  of the  action of $\rist_G(v)$ on the $m$th level of $T_v$ has cardinality greater or equal than $k$. \end{lem}

\begin{proof}
By \cite[Lemma 3.3]{Barth-Grig-hecke} all the orbits of $\rist_G(v)$ in $\partial T_v$ are infinite, so the statement easily follows by compactness of $\partial T_v$.
\end{proof}


%

%

The following is a strengthening of the previous lemma. 

\begin{lem}\label{l-wb-rigid-move-derived}
Let $G\le \aut(T)$ be a weakly branch group. Let $v \in T$, and $k \ge 1$. Then there exists $m=m(v, k)$ such that for every subgroup $D$ of $\rist_G(v)$ of index at most $k$ in $\rist_G(v)$, the action of $D'$ on the $m$th level of $T_v$ does not fix any vertex. 
\end{lem}

\begin{proof}
By Lemma \ref{l-wb-rigid-move-fi}, there exists $m'$ depending only on $v$ and $k$ such that the action of $D$ on the $m'$-th level of $T_v$ does not  fix any vertex. For every $w\in T_v$, the intersection $D\cap \rist (w)$ has index at most $k$ in $\rist_G(w)$. Thus we can find $m''(w)$ such that the action of $\rist_G(v) \cap D$ on the $m''(w)$ level of $T_w$ does not fix any vertex. 
We claim that $m=m'+\max_w m''(w)$ (where the maximum is taken over $w\in T_v$ at the $m'$ level of $T_v$) satisfies the desired conclusion. Indeed let $u\in T_v$ be at level $m$, and let $w\in T_v$ be the unique vertex above $u$ at level $m'$. By construction we can find $g\in D$ such that $g(w) \neq w$, as well as $h\in \rist_G(w) \cap D$ such that $h(u) \neq u$. Then the commutator $[g, h]$ will satisfy $[g, h](u)= h(u) \neq u$, showing that $D'$ does not fix $u$. \qedhere

\end{proof}

The following proposition is the core of the proof of Theorem \ref{t-wb-structure-lsc}. This is the part of the proof that is based on the results of Section \ref{sec-confined}.  

\begin{prop} \label{prop-urs-moves-rist}
Let $G$ be a subgroup of $\aut(T)$, $X$ a minimal compact $G$-space, and $\Phi\colon X\to \sub(G)$ a lower semi-continuous $G$-map. For every $x\in X$ and every $\xi \in \partial T$ that is not fixed by $\Phi(x)$, there exist a vertex $w$ above $\xi$ and a clopen neighbourhood $Z$ of $x$ in $X$ such that $\Phi(z)$ contains $\rist_G(w)'$ for every $z\in Z$. 
\end{prop}

\begin{proof}
We choose $h \in \Phi(x)$ and $v\in T$ above $\xi$ such that $h(v) \neq v$. Let $n$ be the level of $v$ and $L$ the stabilizer of level $n$ in $G$. Let $Z$ be the closure of the $L$-orbit of $x$ in $X$. Then $Z$ is a clopen subset of $X$ and $L$ acts minimally on $Z$ (Lemma \ref{l-minimal-fi}). By lower semi-continuity of $\Phi$, the set $U=\{z\in Z \colon h\in \Phi(z)\}$ is open in $Z$, so that by minimality and compactness there exist $\gamma_1, \ldots, \gamma_r \in L$ such that $Z = \bigcup_i \gamma_i (U)$.

Fix $z \in Z$. It follows that for every $g \in L$ there is $i$ such that $g\Phi(z)g^{-1}=\Phi(g(z))$ contains $\sigma_i:=\gamma_i h \gamma_i^{-1}$. That is,  the set $P=\{\sigma_1, \ldots, \sigma_r\}$ is confining for $(\Phi(z), L)$. Note that by construction we have $\sigma_i(v)=h(v)$ for every $i=1,\ldots, r$. Thus if we set $\Omega_{\sigma_i}= \partial T_v$ for every $i$, conditions (C1) and (C2) of Definition \ref{def-C1-C2} are satisfied. By applying Proposition \ref{p-first-step} with $n=1$ to the pair $(\Phi(z), L)$, we deduce that there exists a subgroup $A_z$ of $\Phi(z)$ that fixes $v$ and $h(v)$, acts trivially on the complement of $\partial T_v \sqcup \partial T_{h^{-1}(v)}$, and such that $D(A_z):= p_{\partial T_v}(A_z)$ is a subgroup of $\rist_G(v)$ of index at most $r$.  Note that since $D(A_z)$ has finite index in $\rist_G(v)$,  Lemma \ref{l-wb-rigid-move-derived} implies that the group $D(A_z)'$ does not fix the point $\xi$.  As a consequence, there exists $d\in D(A_z)$ such that $\xi, d(\xi), d^2(\xi)$ are pairwise distinct (indeed otherwise the image of the permutation representation of $D(A_z)$ on the orbit of $\xi$ would be an abelian 2-group, so that $D(A_z)'$ would fix $\xi$).  Choosing $a\in A_z$ such that $p_{\partial T_v}(a)=d$, we see that $A_z$ (and hence $\Phi(z)$) also contains elements with this property.  In particular this is true for $z=x$, so that  we could have chosen to begin with the  element $h$ and the vertex $v$ so that $v, h(v),h^2(v)$ are pairwise distinct. In the sequel we assume that this is the case.

Now we can apply the argument of the previous paragraph with $h^2$ instead to $h$, and we find for every $z\in Z$  a subgroup $B_z\le \Phi(z)$ that fixes $v$ and $h^{-2}(v)$, acts trivially on the complement of $\partial T_v \cup \partial T_{h^{-2}(v)}$, and such that $D(B_z):= p_{\partial T_v}(B_z)$ is a also a subgroup of index at most $r$ in $\rist_G(v)$.   Note that $D(A_z) \cap D(B_z)$ has index at most $r^2$ in $\rist_G(v)$, and thus contains  a subgroup $C_z$ which is normal in $\rist_G(v)$ and has index at most $(r^2)!$ (observe that although $C_z$ might depend non-trivially on $z\in Z$, the bound on its index does not). For every $c_1, c_2\in C_z$, we choose $a \in A_z, b \in B_z$ such that $p_{\partial T_v}(a)=c_1$ and $p_{\partial T_v}(b)=c_2$. Using that $a,b$ are supported respectively in $\partial T_{v} \sqcup \partial T_{h^{-1}(v)}$ and $\partial T_{v} \sqcup \partial T_{h^{-2}(v)}$, and since $h(v)\neq  h^2(v)$, it follows that $[c_1, c_2] = [a, b]\in \Phi(z)$. Since $c_1, c_2$ were arbitrary, this shows that the derived subgroup $C_z'$ is contained in $\Phi(z)$ for every $z\in Z$. Note also that $C_z'$ is still normal in $\rist_G(v)$. 
 
Since the index of $C_z$ in $\rist_G(v)$ is bounded uniformly on $z$,  by Lemma \ref{l-wb-rigid-move-derived} we can find a vertex $w$ above $\xi$ and below $v$ which is moved by $C_z'$ for every $z\in Z$. Applying Lemma \ref{lem-double} to $C_z'$, we have that $\rist_G(w)'\le C_z' \le \Phi(z)$, concluding the proof.  \qedhere 
\end{proof}

\begin{prop} \label{p-fix-usc}
	Let $G$ be a subgroup of $\aut(T)$, $X$ a minimal compact $G$-space and $\Phi\colon X\to \sub(G)$ a lower semi-continuous $G$-map. Then the map $x\mapsto \fix(\Phi(x))$ is continuous.
\end{prop}

\begin{proof}
	This map is upper semi-continuous by Lemma \ref{lem-compos-fix-lsc-usc}, so we only have to prove lower semi-continuity. So given a vertex $v\in T$, we shall prove that the set $\mathcal{U}_v$ of points $x \in X$ such that $\Phi(x)$ fixes a point in $\partial T_{v}$ is open in $X$. Let $n$ be the level of $v$ in $T$, and $x \in \mathcal{U}_v$. Since the level stabilizer $\st_G(n)$ preserves $\partial T_{v}$, it is clear that $gx \in \mathcal{U}_v$ for all $g \in \st_G(n)$. Since $\st_G(n)$ has finite index in $G$, the closure in $X$ of the orbit of $x$, contains an open neighbourhood of $x$ (Lemma \ref{l-minimal-fi}). But by upper semi-continuity, the subset $\mathcal{U}_v$ is closed in $X$. It follows that $\mathcal{U}_v$ contains an open neighbourhood of $x$, and hence $\mathcal{U}_v$ is open, as desired. \qedhere
\end{proof}

We are now ready to prove Theorem \ref{t-wb-structure-lsc}.

\begin{proof}[Proof of Theorem \ref{t-wb-structure-lsc}]
	For simplicity in the proof we write $C_x = \fix(\Phi(x))$ for $x \in X$. The continuity of the map $ x\mapsto C_x$ has been established in Proposition \ref{p-fix-usc}. In particular the image $F_X$ of this map is closed in $\F(\partial T)$. We now construct the map $\sigma: F_X \to \Pi(T)$ satisfying the desired properties. 

For every vertex $v\in T$, the set $U_v=\{C\in F_X\colon C \cap \partial T_v=\varnothing\}$ is clopen in $F_X$. Let us denote by $O_v=U_v \setminus U_{v'}$, where $v'$ is the vertex above $v$: $O_v$ is the clopen subset of $F_X$ consisting of all $C\in F_X$ such that $\partial T_v \cap C=\varnothing$ and so that $\partial T_v$ is not strictly contained in another cylinder subset with this property. We also denote by $O_v^X=\{x\in X \colon C_x \in O_v\}\subset X$, i.e.  the preimage of $O_v$ under $x\mapsto C_x$, which is a clopen subset of $X$ by continuity of $x\mapsto C_x$. Finally for $v\in T$ let us denote by $\L_{v}(n)$ the $n$th level of the subtree $T_v$.

Note that for $x\in O_v^X$, the group $\Phi(x)$ does not admit fixed points in $\partial T_v$ by definition. Thus by Proposition \ref{prop-urs-moves-rist} and by compactness of $\partial T_v$,  we can find a neighbourhood $V_x\subset O_v^X$ of $x$, and an integer $n({v, x})$ such for every $z\in V_x$, the group $\Phi(z)$ contains $\bigoplus_{w\in \L_v({n(v, x)}}  \rist_G(w)' $. Using now compactness of the set $O_v^X$, we can find a single integer $n(v)$ such that $\bigoplus_{w\in \L_v({n(v)})}  \rist_G(w)' \le \Phi(x)$ for every $x\in O_v^X$.

For $C\in F_X$ we denote by $\Omega_C$ the set of vertices $v$ such that $C\in O_v$, and we set
\[\sigma(C) := \bigcup_{v \in \Omega_C} \L_v({n(v)}).\]
It follows from the definition and the fact that the sets $O_v$ are clopen in $F_X$ that the map $C \mapsto \sigma(C)$ is indeed continuous. Moreover for $C\in F_X$ we have 
\[\partial T \setminus C= \bigsqcup_{v \in \Omega_C} \partial T_v= \bigsqcup_{v \in \Omega_C} \bigsqcup_{w\in \L_v(n(v))} \partial T_w= \bigsqcup_{w\in \sigma(C)} \partial T_w,\] so that the map $\sigma$ is indeed a section of $\pi$, i.e.\ satisfies $\pi \circ \sigma = id$. Furthermore from the definition of $\L_v(n
(v))$ it is clear that $\rist_G(\sigma(C_x))' \le \Phi(x)$, so we have proved all the desired properties. \qedhere 
\end{proof}

\subsection{Factor maps to profinite $G$-spaces} \label{subsec-factor-profinite}

Theorem \ref{t-wb-structure-lsc} can be applied by choosing $\Phi\colon X\to \sub(G)$ to be the germ-stabilizer map $\Phi(x)=G^0_x$, with $X$  any minimal compact $G$-space. Of course, its conclusion is interesting only when the germ-stabilizers of the action are non-trivial, i.e. when the action is not topologically free.   This provides information on the structure of non-topologically free minimal compact $G$-spaces. The main consequence is that any non-topologically free minimal action that is faithful must factor onto a non-trivial profinite $G$-space, namely a subspace of $\F(\partial T)$:

\begin{cor} \label{c-wb-actions}
	Let $G\le \aut(T)$ be a weakly branch group, and $X$ be a compact  minimal $G$-space. Then at least one of the following hold:
	\begin{enumerate}[label=\roman*)]
		\item  \label{i-wb-nf} The action of $G$ on $X$ is not faithful.
		\item \label{i-wb-free} The action of $G$ on $X$ is topologically free.
		\item \label{i-wb-factor} There exists a factor map from $X$ to a compact minimal subset  $F_X\subset \F(\partial T)$, with $|F_X|\ge 2$. In particular, $X$ factors onto a non-trivial profinite $G$-space. 
	\end{enumerate}
\end{cor}

\begin{proof}
	We apply Theorem \ref{t-wb-structure-lsc} to the map $\Phi(x)=G^0_x$. Letting $C_x$ be the set of fixed points of $G^0_x$, we have that  $F_X=\{C_x \colon x\in X\}$ is a factor of $X$, and is a compact minimal $G$-invariant subset of $\F(\partial T)$. If $|F_X|\ge 2$, then we are in case \ref{i-wb-factor}. Assume that $|F_X|=1$.  Since the action of $G$ on $\partial T$  is minimal, the only fixed point for the action of $G$ on $\F(\partial T)$ are $\partial T$ and $\varnothing$, so that $F_X=\{\partial T\}$ or $F_X=\{\varnothing\}$. In the first case, we have that for every $x\in X$ the subgroup $G^0_x$ must fix the whole $\partial T$, so that $G^0_x=\{1\}$. Thus, the action is topologically free and case \ref{i-wb-free} holds. If instead $F_X=\varnothing$, then there exists a collection of vertices $\V:= \V_{\varnothing}\in \Pi(T)$ such that $\partial T= \bigsqcup_{v\in \V} \partial T_v$ and $\rist_G(\V)' \le G_x^0$ for every $x\in X$. From this we deduce that there exists $n$ such that the normal subgroup $\rist_G(n)'$ acts trivially on $X$, so case \ref{i-wb-nf} holds. \qedhere
	
\end{proof}
For a compact $G$-space $X$, the existence of a factor map to a finite $G$-space is a quite restrictive condition. For example, it has the following consequence. Recall that the action of a group $G$ on a compact space $X$ is \textbf{proximal} if  for every pair of points $x, y\in X$ there exists a net $(g_i)$ in $G$ such that $(g_ix)$ and $(g_iy)$ converge to the same limit in $X$. The action of $G$ on $X$ is \textbf{weakly mixing} if the diagonal $G$-action on $X \times X$ is topologically transitive.

\begin{cor} \label{cor-wb-prox-wm}
	Let $G$ be a weakly branch group, and $X$ be a minimal compact $G$-space on which the $G$-action is faithful. If the action of $G$ on $X$ is proximal, or weakly mixing, then it is topologically free. 
\end{cor}
\begin{proof}
	The properties of proximality and weak mixing pass to factors, and a non-trivial profinite $G$-space never satisfies them, so the statement follows from Corollary \ref{c-wb-actions}.
\end{proof}

The profinite factor $F_X$ in case \ref{i-wb-factor} of Corollary \ref{c-wb-actions} may or may not be finite. The following statement provides an interpretation of when it is. Roughly speaking, if $F_X$ is finite, then a combination of cases \ref{i-wb-nf} and \ref{i-wb-free} in Corollary \ref{c-wb-actions} must hold, in the sense that there exists a clopen partition of $X$ that is invariant under a normal subgroup $\rist_G(n) \simeq \prod_{v \in \L(n)} \rist_G(v)$ and such that the action of $\rist_G(v)$ on each piece of this partition is either not faithful or topologically free. See \S \ref{s-iet} for an application of this statement.

\begin{cor} \label{c-infinite-profinite-factor}
	Let $G\le \aut(T)$ be a weakly branch group. For every minimal compact $G$-space $X$, one of the following hold.
	
	\begin{enumerate}[label=\roman*)]
		\item There exists a continuous $G$-map from $X$ to an infinite closed minimal invariant subset $F_X\subset \F(\partial T)$. In particular $X$ factors onto an infinite  profinite $G$-space.

		\item \label{i-finite-factor} There exist a clopen partition $X=X_0\sqcup \cdots \sqcup X_k$ and $n\ge 1$ and  such that each $X_i$ is invariant under the level stabilizer $\st_G(n)$, and for every vertex $v\in \L(n)$ and  $i=1, ..., k$, exactly one of the following holds:
		\begin{itemize}[label=--]
			\item the derived subgroup $\rist_G(v)'$ acts trivially on $X_i$;
			\item  $\partial T_v \subset \fix(G_x^0)$ for every $x\in X_i$. In particular the action of $\rist_G(v)$ on $X_i$ is topologically free.
		\end{itemize}
	\end{enumerate}
	
\end{cor}

\begin{proof} The proof is similar to the proof of Corollary \ref{c-wb-urs-infinite}. We apply Theorem \ref{t-wb-structure-lsc} to the map $\Phi(x)=G^0_x$. If $F_X$ is finite, then it consists of clopen subsets, and as in the proof of Corollary \ref{c-wb-urs-infinite} we can find a level $n$, and for every $x\in X$ a partition $\L(n)=\P_x^+\sqcup \P_x^-$ such that for every $x\in X$ we have $\rist_G(v)'\le G^0_x$ for $x\in \P_x^{+}$, and $\partial T_v \subset \fix(G_x^0)$ if $x\in \P_x^-$. We then let $X=X_0\sqcup \cdots \sqcup X_k$ be the partition of $X$ into fibers of the map that to $x$ associates $(\P_x^-, \P_x^+)$, which is continuous. Each fiber of this map is clearly $\st_G(n)$ invariant (hence $\rist_G(n)$-invariant)  and satisfies the conclusion by construction. Note that if $(v, i)$ are such that $\partial T_v \subset \fix(G^0_x)$ for $x\in X_i$, then $\rist_G(v) \cap G^0_x=\{1\}$, so that the action of $\rist_G(v)$ on $X_i$ is topologically free.  \qedhere 
\end{proof}

\subsection{Micro-supported actions} \label{subsec-microsupp}

As an another application of Theorem \ref{t-wb-structure-lsc}, let us show how it recovers the following reconstruction theorem proven in \cite{Lav-Nek} (see \cite[Th. 2.10.1]{Nek-book} for the following version). 

\begin{cor}[Lavreniuk-Nekrashevych] \label{cor-reconstr-wb}
Suppose that a group $G$ admits two faithful and weakly branch actions on rooted trees $T_1$ and $T_2$. Then there exists a $G$-equivariant homeomorphism $\partial T_1 \to \partial T_2$.
\end{cor}

\begin{proof}
The map $\phi_1: \partial T_2 \to \F(\partial T_1)$, $\xi \mapsto \fix_{\partial T_1}(G_\xi^0)$, is lower semi-continuous, so that we can invoke Theorem \ref{t-wb-structure-lsc}. Fix $\xi \in \partial T_2$.  If $G_\xi^0$ does not fix any point in $\partial T_1$, then by the theorem there would exist a finite set $\V$ of independent vertices  such that the disjoint union of $\partial (T_{1})_v$ for $v \in \V$ is equal to $\partial T_1$, and $\rist_G(v)' \leq G_\xi^0$ for all $v \in \V$. That would imply that $G_\xi^0$ contains the derived subgroup of the rigid stabilizer in $G$ of some level $n$ in $T_1$, which is impossible. Hence there exists $x \in \partial T_1$ such that $G_\xi^0 \leq G_x$. By Lemma \ref{lem-wb-xi1xi2} we have $G_x \leq G_\xi$ (because otherwise $G_x$ would contain a non-trivial normal subgroup of $G$), and again by the same lemma we have $G_\xi \leq G_x$. Hence $G_x = G_\xi$, and it follows that the point $x$ is the only point in $\partial T_1$ that is fixed by $G_\xi$. So we have shown that for every $\xi \in \partial T_2$, there exists a unique $x \in \partial T_1$ that is fixed by $G_\xi$.

Assume for a moment that the group $G$ is countable. That assumption implies that there is a dense $G_\delta$-subset of regular points in $\partial T_2$, ie.\ points $\xi$ such that $G_\xi = G_\xi^0$. For such a point we have that $G_\xi^0$ fixes a unique point in $\partial T_1$ by the above paragraph.  Together with continuity of the map $\phi_1$, this implies that $\phi_1$ actually takes values in $\partial T_1$, so that $\phi_1$ is a continuous $G$-map  $\phi_1: \partial T_2 \to \partial T_1$. Now by symmetry we also have $\phi_2: \partial T_1 \to \partial T_2$, so that $\psi = \phi_1 \circ \phi_2: \partial T_1 \to \partial T_1$ is an endomorphism of the $G$-space $\partial T_1$. Since the $G$-action on $\partial T_1$ is micro-supported, $\psi$ must be the identity. So $\phi_1, \phi_2$ are homeomorphisms, as desired.

If now $G$ is not countable, then we can find a countable subgroup $\Gamma$ of $G$ such that the actions of $\Gamma$ on $T_1$ and $T_2$ are weakly branch (indeed since the vertex sets of $T_1, T_2$ are countable, one can find a countable subset $S$ of $G$ which contains non-trivial elements in all rigid stabilisers and such that $\Gamma:=\langle S\rangle$ is level-transitive on both trees). By the same argument as above we can find $\xi$ such that $\Gamma_\xi^0$ fixes a unique point in $\partial T_1$. But then $G_\xi^0$ also fixes a unique point in $\partial T_1$ since $\Gamma_\xi^0 \leq G_\xi^0$, and hence the previous argument applies for $G$.
\end{proof}

Our present goal is now to explain a result about general micro-supported actions of weakly branch groups. We need some terminology. Let $G$ be a group, and $X,Y$ compact $G$-spaces. An extension  $\pi: Y \to X$ is \textbf{highly proximal} if for every non-empty open subset $U \subseteq Y$, there exists $x \in X$ such that $\pi^{-1}(x) \subseteq U$. Equivalently, the image of every proper closed subset of $Y$ is a proper subset of $X$. When $X,Y$ are minimal, $\pi: Y \to X$ is highly proximal if and only if for every $x \in X$, the fiber $\pi^{-1}(x)$ is compressible, in the sense that there exists $y\in Y$ such that for every neighbourhood $V$ of $y$, there is $g \in G$ such that $g(\pi^{-1}(x)) \subset V$ \cite{Ausl-Glasn}. We say that two compact $G$-spaces $X_1,X_2$ are \textbf{highly proximally equivalent} if $X_1$ and $X_2$ admit a common highly proximal extension \cite{Ausl-Glasn}.

For a compact $G$-space $X$, we denote by $\tilde{X}$ the Stone space of the Boolean algebra $R(X)$ of regular open subsets of $X$. Note that $\tilde{X}$ is also a compact $G$-space. The map $\pi: \tilde{X} \to X$, which associates to every ultrafilter on $R(X)$ its limit in $X$, is a continuous and surjective $G$-map, and is highly proximal \cite[Th.\ 3.2]{Gleason1958}.

Recall that if a group $G$ admits faithful and micro-supported actions on two compact spaces $X_1,X_2$, then a theorem of Rubin \cite{Rubin-loc-mov-book} asserts that there exists a $G$-equivariant isomorphism between $R(X_1)$ and $R(X_2)$, or equivalently a $G$-equivariant homeomorphism between $\tilde{X_1}$ and $\tilde{X_2}$, so that  $X_1,X_2$ are highly proximally equivalent. Since a highly proximal extension of a micro-supported $G$-space remains micro-supported (see \cite[Proposition 2.3]{CLB-commens}), Rubin theorem implies that for every group admitting a faithful and micro-supported action on a compact space, there exists a (necessarily unique) universal compact $G$-space on which the $G$-action is faithful and micro-supported which factors onto every compact $G$-space on which the $G$-action is faithful and micro-supported. Hence there always exists a largest object among faithful and micro-supported $G$-spaces.

The goal of this paragraph is to exhibit sufficient conditions ensuring that a similar behaviour happens \enquote{at the bottom}, in the sense that there exists a smallest object among faithful and micro-supported $G$-spaces (see Corollary \ref{cor-micro-hp}). This will apply to weakly branch groups, providing a companion of Corollary \ref{cor-reconstr-wb} (see Corollary \ref{cor-wb-micro-hp}).

Particular instances of the two following  results already appeared in \cite[\S 4]{LBMB-sub-dyn}. Recall from \S \ref{subsec-URS-conf} that if $X$ is a minimal compact $G$-space, $\mathcal{S}_G(X)$ is the stabilizer URS associated to $X$. 

\begin{lem} \label{lem-Hx-usc}
	Let $G$ be a group, $X$ a minimal compact $G$-space and $\H$ a URS of $G$. For $x \in X$, let $\H_x$ be the set of $H \in \H$ such that $G_x^0 \leq H \leq G_x$. Then:
	\begin{enumerate}[label=\roman*)]
		\item  The map $x \mapsto \H_x$ is upper semi-continuous.
		\item If $\H = \mathcal{S}_G(X)$ then $\H = \cup_x \H_x$.
	\end{enumerate}
\end{lem}

\begin{proof}
The verification of  upper semi-continuity is routine, and we leave it to the reader. Suppose $\H = \mathcal{S}_G(X)$, and let $H \in \H$. Then there exists $(x_i)$ such that $G_{x_i}$ converges to $H$, and we easily verify that $H \in \H_x$ for every accumulation point $x$ of $(x_i)$. 
\end{proof}

\begin{prop} \label{prop-sameURS-factor}
	Let $G$ be a group and $X$ a minimal compact $G$-space on which the $G$-action is faithful and with the property that for every $x_1 \neq x_2 \in X$, the subgroup generated by $G_{x_1}^0$ and $G_{x_2}^0$ contains a non-trivial normal subgroup of $G$. Let $Y$ a minimal compact $G$-space such that $\mathcal{S}_G(Y) = \mathcal{S}_G(X)$. Then $Y$ factors onto $X$.
\end{prop}

\begin{proof}
	Let $\H = \mathcal{S}_G(X) = \mathcal{S}_G(Y)$. According to Lemma \ref{lem-Hx-usc} we have $\H = \cup_{x\in X} \H_x = \cup_{y\in Y} \H_y$, and by our assumption on the $G$-action on $X$ we have that the sets $\H_x$ are pairwise disjoint. 
	
	We claim that for every $y$, there exists $x$ (necessarily unique) such that $\H_y \subseteq \H_x$. Indeed if $H,K \in \H_y$ are such that $H \in \H_{x_1}$ and $K \in \H_{x_2}$, then $G_y$ contains $G_{x_1}^0$ and $G_{x_2}^0$ because it contains $H$ and $K$, and hence we deduce that $x_1 = x_2$. We define $\phi(y) = x$.
	
	That $\phi$ is $G$-equivariant is clear. In order to see that it is continuous, it is enough to see that whenever $(y_i)$ converges to $y$ and $(\phi(y_i))$ converges to $x$, then $\phi(y) = x$. Without loss of generality we may assume that $\H_{y_i}$ and $\H_{\phi(y_i)}$ converge respectively to $\K$ and $\L$, and clearly $\K \subseteq \L$. Moreover by upper semi-continuity of $x \mapsto \H_x$ and $y \mapsto \H_y$ (Lemma \ref{lem-Hx-usc}), we have $\K \subseteq \H_y$ and $\L \subseteq \H_x$, so that $\K \subseteq \H_y \cap \H_x$. In particular $\H_y \cap \H_x$ is not empty, which by definition implies $\phi(y) = x$, as desired. So the proof is complete.
\end{proof}


Recall from  \S  \ref{subsec-URS-conf} that if $\H, \K$ are closed subsets of $\sub(G)$, we write $\H \preccurlyeq \K$ if there exists $H \in \H$ and $K \in \K$ such that $H \leq K$.

\begin{lem} \label{lem-hp-sameURS}
Two highly proximally equivalent minimal $G$-spaces give rise to the same stabilizer URS.
\end{lem}

\begin{proof}
It is enough to show that $\pi: Y \to X$ is highly proximal, then  $\mathcal{S}_G(X) = \mathcal{S}_G(Y)$. Write $\mathcal{T}_G(X) = \overline{\left\lbrace G_x : x \in X\right\rbrace }$ and $\mathcal{T}_G(Y) = \overline{\left\lbrace G_y : y \in Y\right\rbrace }$, and recall that $\mathcal{S}_G(X)$ and $\mathcal{S}_G(Y)$ are the unique minimal closed $G$-invariant subsets of respectively $\mathcal{T}_G(X)$ and $\mathcal{T}_G(Y)$. Since $Y$ is an extension of $X$, we have $ \mathcal{T}_G(Y) \preccurlyeq \mathcal{T}_G(X)$. By Lemma 2.13 in \cite{LBMB-sub-dyn} this implies $ \mathcal{S}_G(Y) \preccurlyeq \mathcal{S}_G(X)$. Now let $x$ in $X$. Since $\pi: Y \to X$ is highly proximal, the closure of the $G$-orbit of $\pi^{-1}(\left\lbrace x\right\rbrace )$ in $\F(Y)$ contains a singleton $\left\lbrace y\right\rbrace$. By upper semi-continuity of the stabilizer map, this implies $ \mathcal{T}_G(X) \preccurlyeq \mathcal{T}_G(Y)$. Again by Lemma 2.13 in \cite{LBMB-sub-dyn} we deduce that $ \mathcal{S}_G(X) \preccurlyeq \mathcal{S}_G(Y)$. Since $\preccurlyeq$ is an order on the set of URSs of $G$ \cite[Cor.\ 2.15]{LBMB-sub-dyn}, the equality $ \mathcal{S}_G(X) = \mathcal{S}_G(Y)$ follows.
\end{proof}

\begin{cor} \label{cor-micro-hp}
	Let $G$ be a group and $X$ a faithful and micro-supported minimal compact $G$-space, with the property that for every $x_1 \neq x_2 \in X$, the subgroup generated by $G_{x_1}^0$ and $G_{x_2}^0$ contains a non-trivial normal subgroup of $G$. Then the faithful and micro-supported minimal compact $G$-spaces are exactly the highly proximal extensions of $X$.
\end{cor}

\begin{proof}
It follows easily from the definitions that if $Y$ is a highly proximal extension of $X$, then $Y$ is a faithful and minimal $G$-space. Moreover by  \cite[Proposition 2.3]{CLB-commens} $Y$ is also micro-supported. Conversely, let $Y$ be a faithful, minimal and micro-supported $G$-space. According to Rubin theorem \cite{Rubin-loc-mov-book}, $X$ and $Y$ are highly proximally equivalent. By Lemma \ref{lem-hp-sameURS} we have $\mathcal{S}_G(X) = \mathcal{S}_G(Y)$. Hence Proposition \ref{prop-sameURS-factor} applies and provides a factor map $\pi: Y \to X$, and we shall check that this map is highly proximal. Let $U$ be a non-empty open subset of $Y$. Since the rigid stabilizer $\rist_G(U)$ is non-trivial and the $G$-action of $X$ is faithful, we can find $x \in X$ and $g \in \rist_G(U)$ such that $g(x) \neq x$. Hence the fiber $\pi^{-1}(x)$ is such that $g(\pi^{-1}(x))$ is disjoint from $\pi^{-1}(x)$. Since $g$ is supported in $U$ we must have $\pi^{-1}(x) \subset U$, and hence $\pi: Y \to X$ is indeed a highly proximal extension. 
\end{proof}

This applies to weakly branch groups by Lemma \ref{lem-wb-xi1xi2}, and we obtain:

\begin{cor} \label{cor-wb-micro-hp}
Let $G\le \aut(T)$ be a weakly branch group. Then the faithful and micro-supported minimal compact $G$-spaces are exactly the highly proximal extensions of $\partial T$.
\end{cor}

%

%


%

\section{Groups acting on rooted trees II: rigidity of actions with small Schreier graphs} \label{sec-trees-graphs-actions}

In this section we give give further applications of the results of the previous sections to  rigidity results for actions of finitely generated (weakly) branch groups.  A common feature of all results in this section is that they rely on the study of the geometry of graphs of actions. We refer the reader to \S \ref{s-preliminaries} for the necessary terminology concerning graphs of group actions. The general principle is that if a finitely generated branch group $G\le \aut(T)$ acts on a set $X$ with sufficiently ``nice''  graphs, then the action of $X$ must be tightly related to the action on $\partial T$ (in a sense that will be specified in each situation). This is used to to prove rigidity results for certain types of actions and for embeddings of $G$ into other groups of homeomorphisms and of automorphisms of rooted trees. 

\subsection{Growth of actions of branch groups} \label{subsec-growth}



Let $\Gamma$ be a graph of bounded degree (not necessarily connected). The \textbf{uniform growth} of $\Gamma$ is the function 
\[\vol_\Gamma(n)=\sup_{v\in \Gamma} |B_\Gamma (v, n)|,\]
where $B_\Gamma(v, n)$ is the ball of radius $n$ around $v$. If $G = \left\langle  S  \right\rangle $ is a finitely generated group and $X$ a $G$-set, the \textbf{growth of the action} $\vol_{G, X}$ of $G$ on $X$ is the uniform growth of the graph $\Gamma(G,X)$ of the action of $G$ on $X$ (with respect to the generating set $S$):
\[\vol_{G, X}(n):= \vol_{\Gamma(G, X)}(n).\] 

Given two functions $f, g\colon \N\to \N$, we write $f\preceq g$ if there is $C > 0$ such that $f(n)\le Cg(Cn)$, and $f\sim g$ if $f\preceq g$ and $g \preceq f$. Since two finite generating subsets $S,S'$ give rise to bi-Lipschitz graphs, the equivalence class of $\vol_{G,  X}$ with respect to $\sim$ does not depend on the choice of $S$. This justifies that we omit $S$ in the notation. Note that we do not require the action of $G$ on $X$ to be transitive; in particular the function $\vol_{G,  X}$ can be unbounded even if every individual $G$-orbit in $X$ is finite. 

The following lemma is immediate from the definitions.

 \begin{lem} \label{l-growth-subgroup}
 Let $G$ be a finitely generated group acting on a set $X$ and $H\le G$ be a finitely generated subgroup. Then $\vol_{H, X} \preceq \vol_{G, X}$. 
 \end{lem}

If $X$ is a compact space and the action of $G$ on $X$ is minimal, the growth of the action is equal to the growth of each of its orbital graphs:

\begin{lem} \label{l-growth-minimal-action}
	Let $G$ be a finitely generated group acting minimally on  a Hausdorff space $X$.   Then $\vol_{G, X}=\vol_{\Gamma(G, x)}$ for every $x \in X$. In particular $\vol_{\Gamma(G, x)}$ does not depend on the point $x\in X$.
\end{lem}

\begin{proof}
	It is obvious that for every $n\in \N$ we have $\vol_{\Gamma(G, x)}(n) \leq \vol_{G, X}(n)$ for every $x\in X$, so let us prove the converse. Set $m=\vol_{G, X}(n)$. By definition, we can find a point $y\in X$ such that $|B_{\Gamma(G, y)}(n, y)|=m$. Thus there exists $g_1, \ldots, g_m\in G$ with word length $\le n$ in the generating set $S$ such that $g_1(y),\ldots, g_m(y)$ are pairwise distinct. Choose an open neighbourhood $U$ of $y $ such that $g_1(U),\ldots, g_m(U)$ are pairwise disjoint. By minimality there exists $z\in U$ in the same orbit as $y$. Then the points $g_1(z), \ldots, g_m(z)$ are pairwise distinct. This shows that $|B_{\Gamma(G, x)}(n, z)|\ge m$, so that $\vol_{\Gamma(G, x)}(n) \geq m= \vol_{G, X}(n)$. \qedhere 
	\end{proof}

The main result of this subsection is the following:

\begin{thm} \label{thm-branch-wobb}
	Let $G \leq \aut(T)$ be a finitely generated  branch group. Then for every $G$-set $X$ on which $G$ acts faithfully, the growth $\vol_{G, X}$ satisfies $\vol_{G, X} \succeq \vol_{G, \partial T}$.  
\end{thm}

\begin{proof}
Let $\X\subset \sub(G)$ be the closure of the stabilisers $G_x, x\in X$. For every $x\in X$, the growth $\vol_{\Gamma(G, x)}$ of the orbital graph associated to $x$ satisfies $\vol_{\Gamma(G, x)} \preceq \vol_{G, X}$, with constants uniform on $x$. For $H\in \X$, we can find a sequence of points $(x_n)$ such that the pointed graphs $\Gamma(G, x_n)$ converge to $\Gamma(G, H)$, and thus every finite ball in $\Gamma(G, H)$ appears in $\Gamma(G, x_n)$ for  some $n$ large. Therefore $\vol_{\Gamma(G, H)} \preceq \vol_{G, X}$.
	
	Let $H \in \X$ that is a URS of $G$ (recall that we say that a subgroup $H\le G$ is a URS if its orbit-closure is a URS). By Corollary \ref{cor-urs-parabolic}, either there exists a point $\xi\in \partial T$ such that $H$ fixes $\xi$, or there exists a level $n$ such that $H$ contains $\rist_G(n)'$. Assume towards a contradiction that no URS $H \in \X$ fixes a point in $\partial T$. Given an arbitrary subgroup $K\in \X$, the closure of the $G$-orbit  of  $K$ contains a URS, so we can find a sequence $(g_i)$ such that $(g_iKg_i^{-1})$ converges to some $H\in \X$ which is a URS. By our assumption, there exists $n$ such that $\rist_G(n)'\le H$. But $\rist_G(n)'$ is finitely generated by Theorem \ref{t-Francoeur}, so that the set of subgroups that contain it is open in $\sub(G)$. We deduce that $g_iK g_i^{-1}$ contains $\rist_G(n)'$ for $i$ large enough, and since $\rist_G(n)'$ is normal, we actually have $\rist_G(n)'\le K$. It follows that if we set  $U_n=\{H\in \X \colon \rist_G(n)' \le H\}$, then $U_{n}\subset U_{n+1}$ and the sets $(U_n)$ form an open cover of $\X$. Thus by compactness there exists $n$ such that $\rist_G(n)'\le K$ for every $K\in \X$. In particular $\rist_G(n)' \le G_x$ for every $x\in X$, and we have reached a contradiction since the action on $X$ is supposed to be faithful. 
	
	Hence it follows that there must exist a URS $H\in \X$ and $\xi  \in \partial T$ such that $H \leq G_\xi$.  This implies that $\vol_{\Gamma(G, H)} \succeq \vol_{\Gamma(G, \xi)}= \vol_{G, \partial T}$ (Lemma \ref{l-growth-minimal-action}). Since $ \vol_{G, X} \succeq \vol_{\Gamma(G, H)}$, this terminates the proof. \qedhere
\end{proof}

\begin{remark}
	Theorem \ref{thm-branch-wobb} admits an equivalent reformulation in terms of embeddings of $G$ into wobbling groups. Recall that the \textbf{wobbling group} of a graph $\Delta$ is the group $W(\Delta)$ of all permutations $\sigma$ of the vertex set of $\Delta$ which have bounded displacement in the sense that $\sup_{v\in \Delta} d_\Delta(\sigma(v), v))<\infty$, where $d_\Delta$ is the simplicial distance. Theorem \ref{thm-branch-wobb} implies that if $G$ is a finitely generated branch group, and if $\Delta$ is a graph such that $\vol_\Delta \nsucceq \vol_{G, \partial T}$, then every homomorphism $\rho\colon G\to W(\Delta)$ has a non-trivial kernel (and thus its image is virtually abelian, by Grigorchuk's Lemma \ref{l-Gri-branch}).  Indeed it is not difficult to see that for every finitely generated subgroup $H\le W(\Delta)$ the graph of the action of $H$ on $\Delta$ is Lipschitz embedded in $\Delta$, and thus its growth is bounded above by $\vol_\Delta$.  
\end{remark}

As a concrete application of Theorem \ref{thm-branch-wobb}, observe that it provides an invariant to show the non-existence of embeddings of $G$ into other groups, using Lemma \ref{l-growth-subgroup}.
\begin{cor} \label{c-homom-growth}
Let $G\le \aut(T)$ be a finitely generated branch group. Let $H$ be a finitely generated group, and assume that $H$ admits a faithful action on a set $X$ such that $\vol_{H, X} \nsucceq \vol_{G, \partial T}$. Then every homomorphism $\rho \colon G\to H$ has virtually abelian image.  
\end{cor}

%

Consider the following property, defined in \cite[Definition 1.12]{MB-graph-germs}.

\begin{defin}
	Let $\lambda \colon \N\to \N$ be a function. A finitely generated group $G$ is said to have property $\fg_\lambda$ if for every action of $G$ on a set $X$ either $\vol_{G, X} \succeq \lambda$, or the image of $G$ in $\sym(X)$ is finite. 
	\end{defin}
 \begin{remark}
 Definition 1.12 in \cite{MB-graph-germs} is formulated in a different terminology, which also applies to non-finitely generated groups, however for finitely generated groups it is equivalent to the previous one by \cite[Prop. 8.1]{MB-graph-germs}
 \end{remark}
It is a well-known observation that if $G$ has Kazhdan's property $(T)$, then it has property $\fg_{\exp}$, where $\exp\colon \N \to \N $ is the exponential function (this remark is attributed to Kazhdan in \cite[Remark 0.5 F]{Gro-asym}). This follows from the fact that if a group $G$ acts on a set $X$ with $\vol_{G, X}(n) \nsim \exp(n)$, then the $G$ action on every orbit preserves an invariant mean (i.e. a finitely additive probability measure). For the same reason, if a group $G$ has property FM  in the sense of Cornulier \cite{Cor-FM}, i.e.\ if every $G$-action with an invariant mean has a finite orbit, then $G$ has property $\fg_{\exp}$ \cite[Th. 7.1]{Cor-FM} (note that property (T) implies property FM). 
It  is shown in \cite[\S 8]{MB-graph-germs} that there exists groups with property $\fg_{\exp}$ but not property  FM, and also that there exists groups that have property $\fg_\lambda$, where $\lambda$ is sharp and varies in a vast class of subexponentially growing functions. These examples are obtained using topological full groups of \'etale groupoids. The following corollary of Theorem \ref{thm-branch-wobb} shows that the class of branch groups is also a source of examples of groups enjoying this property. 

\begin{cor}
Let $G\le \aut(T)$ be a just-infinite finitely generated branch group, and let $\lambda= \vol_{G, \partial T}$. Then $G$ has property $\fg_\lambda$.
\end{cor}

\subsection{Actions  of weakly branch groups with polynomially growing orbits and interval exchange transformations} \label{s-iet}

Recall that using Theorem \ref{t-wb-structure-lsc}, in \S \ref{subsec-factor-profinite} we associated to every minimal and non topologically free action of a weakly branch group $G$ on a compact space $X$, a continuous $G$-equivariant map from $X$ to $\F(\partial T)$. The following result asserts that when the growth of the action of $G$ on $X$ is bounded above by a polynomial, then the image of this map is infinite.

\begin{thm}\label{t-wb-polynomial}
	Let $G \le \aut(T)$ be a finitely generated weakly branch group, and $X$ be a minimal compact $G$-space on which $G$ acts faithfully. Assume that $\vol_{G, X}(n) \preceq n^d$ for some $d\ge 1$. Then there exists a continuous $G$-equivariant map  $q\colon X\to \F(\partial T)$ with infinite image. In particular $X$ factors onto an infinite profinite $G$-space. 
\end{thm}


\begin{proof}
	According to Corollary \ref{c-infinite-profinite-factor} applied to the action of $G$ on $X$, it is enough to show that case \ref{i-finite-factor} in the conclusion of Corollary \ref{c-infinite-profinite-factor} cannot hold in the present situation. We argue by contradiction and assume  that there exist $r \ge 1$ and a clopen partition $X=X_1\sqcup \cdots \sqcup X_k$ that is invariant under $\st_G(r)$ and that satisfies the conclusion \ref{i-finite-factor} of Corollary \ref{c-infinite-profinite-factor}. The subgroup $\rist_G(v)'$ cannot act trivially on every $X_i$, so we can choose $v\in \L(r)$ and $i=1,\cdots, k$ such that $\partial T_v\subset \fix(G_x^0)$ for every $x\in X_i$. We fix $v$ and $i$, and we also choose a point $y\in X_i$ such that $G^0_y=G_y$.
	
	Let $\pi_v\colon \st_G(r) \to \aut(T_v)$ be the map obtained by restricting the action of $\st_G(r)$ on $T_v$, and let $K_v:=\pi_v(\st_G(r))$. Note that since $\st_G(r)$ has finite index in $G$, it is finitely generated, and thus so is $K_v$.  We fix a finite generating subset $S$ of $\st_G(r)$ and consider the corresponding generating subset $\pi_v(S)$ of $K_v$, that we use to compute the growth $\vol_{K_v}$. We will show that $\vol_{K_v}$ must be polynomial, and derive a contradiction. 
	
	Fix $n$ and let $m=\vol_{K_v}(n)$. Let $h_1,\cdots, h_m$ be an enumeration of the ball of radius $n$ in $K_v$. For every $i=1,\ldots, m$ choose a representation  $h_i=\pi_v(s_1\cdots s_\ell)$ as a product of generators of minimal length and set $g_i=s_1\cdots s_\ell$ be the corresponding element of $G$. If  $i\neq j$ then $g_iy\neq g_j y$ since  $g_ig_j^{-1}$ projects to $h_ih_j^{-1}\neq 1$, thus does not fix $\partial T_v$, while $\partial T_v \subset \fix(G^0_y)=\fix(G_y)$. Since $\vol_{G, X}\preceq n^d$, and each $g_i$ has length at most $n$ in the generating set $S$ of $\st_G(r)$, we deduce that 
	\[\vol_{K_v}(n)=m=|\{g_1y,\ldots, g_my\}| \le Cn^d\]
	for some constant $C=C(S)$ independent of $n$. Thus the group $K_v$ has polynomial growth, and hence must be virtually nilpotent by Gromov's theorem \cite{Gro-growth}. But it is easy to check that $K_v$ cannot be virtually nilpotent (alternatively, this also follows from Lemma \ref{lem-rist-FC}). This provides a contradiction, and concludes the proof. \qedhere
\end{proof}

We now describe an application in the setting of interval exchange transformations. An {interval exchange transformation} is a left-continuous permutation of $\R/\Z$, with finitely discontinuity points, and which coincides with a translation in restriction to every interval of continuity. Interval exchange transformations form a group, denoted by $\iet$. 
For a countable subgroup $\Lambda$ of  $\R/\Z$, we also define $\iet(\Lambda)$ to be the subgroup of $\iet$ consisting of elements whose discontinuity points belong to $\Lambda$, and that are given by translations by an element of $\Lambda$ in restriction to every interval of continuity. 

 The dynamics of iterations of one interval exchange transformation is a classical and well-developed topic, see \cite{Via} for a survey. More recently there has  been  interest in the study of actions of more general groups by interval exchanges. A central question in the field is  to understand which finitely generated groups can act faithfully by interval exchanges. It turns out that the subgroup structure of $\iet$ appears to be more restricted than it might look, although few explicit obstructions are currently known. For example, Novak proved that a finitely generated subgroup of $\iet$ cannot have a distorted infinite cyclic subgroup \cite{Nov-disc}, and  Dahmani-Fujiwara-Guirardel \cite{DFG-solv} proved that a finitely generated torsion-free solvable subgroup of $\iet$ is virtually abelian (in contrast, they also show that there are uncountably many non-isomorphic finitely generated solvable subgroups of $\iet$ containing torsion \cite{DFG-solv}). Another subgroup obstruction comes from work of Cornulier \cite{Cor-comm-pc}, which implies that if $G$ is a group with  property FW (see \cite{Cor-comm-pc}), then every homomorphism from $G$ to $\iet$ has finite image.   

Note that every finitely generated subgroup $G$ of $\iet$ is a subgroup of $\iet(\Lambda)$ for some finitely generated dense subgroup $\Lambda$ of $\R/\Z$, so that it is natural to study how  the subgroup structure of $\iet(\Lambda)$ depends on $\Lambda$.  It was proven in \cite{ExtAmen} that if $\rk_\Q(\Lambda)\le 2$, then the group $\iet(\Lambda)$ is amenable, generalising a result of Juschenko and Monod \cite{Ju-Mo} which implies the same result in the case $\rk_\Q(\Lambda)=1$.   A result of the second author in \cite{MB-graph-germs} classifies the pairs $(\Lambda, \Delta)$ such that $\iet(\Lambda)$ can be embedded into $\iet(\Delta)$.

In contrast with these results, several basic questions remain open (see the discussion before Theorem \ref{t-wb-iet-intro}). As an application of Theorem \ref{t-wb-polynomial}, we will prove the following:


\begin{thm}\label{t-wb-iet}
	Let $G \le \aut(T)$ be a finitely generated weakly branch group. Then $G$ does not admit any faithful action on $\R/\Z$ by interval exchange transformations.
\end{thm}


We need some preliminaries on the dynamics of actions of  finitely generated subgroups of $\iet$. Since every finitely generated subgroup of $\iet$ is contained in $\iet(\Lambda)$ for some finitely generated subgroup $\Lambda<\R/\Z$, without loss of generality we will consider finitely generated subgroups of  $\iet(\Lambda)$.  
 
We first recall that although the action of $\iet(\Lambda)$ on $\R/\Z$ is not continuous, the group $\iet(\Lambda)$ can be viewed as a group of continuous transformations via the following well-known ``doubling trick''.  Let $X_\Lambda$ be the space obtained from $\R/\Z$ by  replacing each point $\lambda\in \Lambda\subset \R/\Z$ with two copies $[\lambda]_-, [\lambda]_+$. Endow $X_\Lambda$ with the natural circular order induced by the circular order on $\R/\Z$, by replacing each $\lambda$ with a pair of adjacent elements   $[\lambda]_-, [\lambda]_+$. The topology induced by this circular order turns $X_\Lambda$ into a compact space homeomorphic to a Cantor set. The group $\iet(\Lambda)$ acts on $X_\Lambda$, by letting it act as usual on $(\R/\Z\setminus \Lambda)\subset X_\Lambda$, and extending its action by continuity to points of the form $[\lambda]_\pm$. Explicitly, for every $g\in \iet$ and recalling that $g$ is left-continuous on $\R/\Z$, we have  $g([\lambda]_-)= [g(\lambda)]_-$ and $g([\lambda]_+)=[\tilde{g}(\lambda)]_+$, where $\tilde{g}$ is the unique right-continuous map which coincides with $g$ away from its points of discontinuity.

We will need the following well-known observation, which can be found e.g. in \cite[Lemma 6.3]{DFG-free}. Recall that the \textbf{rational rank}  $\operatorname{rk}_\Q(\Lambda)$ of $\Lambda$ is the rank of its torsion-free part. We also recall from \S \ref{subsec-growth} that $\vol_\Gamma(n)$ is the (uniform) growth of a graph $\Gamma$.

\begin{lem} \label{l-iet-polynomial}
Assume that $\Lambda<\R/\Z$ is finitely generated, and set $d:= \operatorname{rk}_\Q(\Lambda)$. Then for every finitely generated subgroup $H\le \iet(\Lambda)$ and every point $x\in X_\Lambda$, the growth of the orbital graph $\Gamma(H, x)$ satisfies $\vol_{\Gamma(H, x)}(n)\preceq n^d$.  
\end{lem}
  
The following result is due to Imanishi \cite{Ima}. For the following formulation, see \cite[Corollary 2.4]{DFG-solv}.  

\begin{thm}[Imanishi \cite{Ima}, see  Corollary 2.4 in \cite{DFG-solv}] \label{t-Imanishi}
Let $\Lambda\le \R/\Z$ be a countable subgroup and $G\le \iet(\Lambda)$ be finitely generated. Then there exists a unique decomposition $X_\Lambda=W_1\sqcup \cdots \sqcup W_k \sqcup Z$ into  $G$-invariant clopen subsets such that: 
\begin{enumerate}[label=\roman*)]
\item the action of $G$ on each of the sets $W_1,\cdots, W_k$ is minimal;
\item the action of $G$ on  $Z$ factors through a finite quotient of $G$. 
\end{enumerate}
\end{thm}
\begin{remark}
The statement of Corollary 2.4 in \cite{DFG-solv} is given in terms of the actions on $\R/\Z$ but it is straightforward to translate it in terms of the action on $X_\Lambda$  as above.

\end{remark}

The sets $W_i$ will be called the \textbf{minimal components} of $G$, and the set $Z$ its \textbf{periodic component}. We also need the following definition. 
\begin{defin}
A subgroup $G\le \iet(\Lambda)$ is \textbf{unfragmentable} if the following hold
\begin{enumerate}[label=\roman*)]  \item The action of $G$ on its periodic component is trivial. 
 \item Every finite index subgroup $H<G$ acts minimally on each minimal components of $G$ (thus the minimal components of $H$ coincide with those of $G$.)
 \end{enumerate}
\end{defin}
The above definition is given by Dahmani, Fujiwara, and Guirardel in \cite{DFG-solv}, who prove the following improvement of Imanishi's theorem \cite[Theorem 2.11]{DFG-solv}.

\begin{thm}[Dahmani--Fujiwara--Guirardel]
Every finitely generated subgroup $G\le \iet(\Lambda)$ has a finite index subgroup $G_0\le G$  which is unfragmentable.  
\end{thm}

We will use this result via the following corollary. 

\begin{cor} \label{c-iet-profinite-factor}
Let $G$ be a finitely generated subgroup of $\iet(\Lambda)$, and $W_i \subset X_\Lambda$ be a minimal component of $G$. Then $W_i$ cannot factor onto an infinite profinite $G$-space. 
\end{cor}

\begin{proof}
	There exists  $n \geq 1$ such that every subgroup of finite index in $G$ has at most  $n$ mimimal components contained in $W_i$. If $W_i$ has an infinitie profinite factor, then it factors onto a finite  space $Y$ with at least $n+1$ points. Then the kernel $H$ of the action on $Y$ has finite index and preserves all fibers of the factor map, thus must have at least $n+1$ disjoint minimal components in $W_i$, giving a contradiction. \qedhere
\end{proof}


\begin{proof}[Proof of Theorem \ref{t-wb-iet}]
Assume for a contradiction that there exists an embedding of $G$ into $\iet$. Then $G$ actually embeds in $\iet(\Lambda)$ for some finitely generated subgroup $\Lambda<\R/\Z$. Let $X_\Lambda=W_1\sqcup \cdots \sqcup W_k \sqcup Z$ be the decomposition into minimal and periodic components (see Theorem \ref{t-Imanishi}). We claim that there exists a minimal component $W_i$ on which $G$ acts faithfully.  Indeed if this is not the case, and since the action on $Z$ is also not faithful, by Lemma \ref{l-Gri-branch}  we can find $n_1,\cdots, n_k, n_Z$ such that $\rist_G(n_i)'$ acts trivially on $W_i$ for $i=1,\cdots, k$, and $\rist_Z(n_Z)'$ acts trivially on $Z$. Thus $\rist_G(m)'$ acts trivially on $X_\Lambda$ for $m=\max\{n_1, \cdots, n_k, n_Z\}$, contradicting that the action of $G$ is faithful. Without loss of generality, we assume that $G$ acts faithfully on $W_1$.  By Lemma \ref{l-iet-polynomial} and Theorem \ref{t-wb-polynomial}, the action of $G$ on $W_1$ must factor onto an infinite profinite $G$-space. This is in contradiction with Corollary \ref{c-iet-profinite-factor}. \qedhere 
\end{proof}

\begin{remark} \label{rmk-infgen-branch}
The following example shows that finite generation is a necessary assumption in Theorem \ref{t-wb-iet}. Consider the regular tree of words $T$ indexed by the alphabet $E=\{1,\cdots, d\}$ (see \S \ref{s-bounded-automorphisms} for the terminology), and let $G\le \aut(T)$ be the group of finitary automorphisms, i.e.\ automorphisms $g\in \aut(T)$ such that the section $g|_v$ is non-trivial only for finitely many vertices $v\in T$. Then $G$ is a countable locally finite branch group. Writing elements of $\R/\Z$ in base $d$, we obtain a natural surjective continuous map $\pi \colon \partial T \to \R/\Z$. It is not difficult to check that the action of $G$ on $\partial T$ descends through $\pi$  to an action on $\R/\Z$ by piecewise continuous transformations, which identifies $G$ with  a subgroup of the group $\iet(\Lambda)$ for $\Lambda=\Z[1/d]/\Z$. 
\end{remark}

\begin{remark} \label{rmk-grig-group}
For example, Theorem \ref{t-wb-iet} implies that the Grigorchuk group $G$   does not embed in the group $\iet$. 
This can be compared with the fact that the Grigorchuk group shares some dynamical features with subgroups of $\iet$.  For instance the orbital graphs of the natural action of the Grigorchuk group on the boundary of its defining tree have linear growth \cite{Barth-Grig-hecke} (compare with Lemma \ref{l-iet-polynomial}). This is used in  \cite{MB-Gri-full} to observe that it can be embedded in the topological full group of a minimal subshift. As observed in \cite{Cor-Bou}, the group $\iet$ also contains the topological full groups of a family of minimal subshifts, for instance Sturmian subshifts (in our notation these correspond to the subgroups $\iet(\Lambda)$ when $\Lambda=\langle \alpha \rangle$ for an irrational element $\alpha\in \R/\Z$). Moreover, the subshift associated to the Grigorchuk group shares with Sturmian subshifts the property to have linear word complexity (see \cite{MB-Gri-full}). We also point out that the Grigorchuk group admits a natural description as a group of interval exchanges with an infinite number of discontinuities (this is actually the point of view adopted in Grigorchuk's original article  \cite{Gri-growth}).  
\end{remark}

\subsection{Rigidity of embeddings into groups of homeomorphisms via graphs of germs} \label{subsec-rigid-embed}

The goal of this subsection is to prove a rigidity result for embeddings of finitely generated branch groups $G$ into a class of groups of homeomorphisms. We show that if $H$ is a finitely generated group of homeomorphisms of a compact space $X$  and if the {graphs of germs} of $\tilde{\Gamma}(H, x)$ of the action of $H$ satisfy a suitable one-dimensionality condition, then every group embedding $\rho \colon G\to H$ must be spatially realized in the following sense: if $Y\subset X$ is the \emph{essential support}  of $\rho(G)$ (defined below), then the action of $G$  on $Y$ induced by $\rho$ factors onto its natural action on $\partial T$.  See Theorem \ref{t-actions-leud}.

In \S \ref{s-strongly-bounded} we illustrate Theorem \ref{t-actions-leud} with a class of groups that we call strongly bounded type, essentially defined by Juschenko--Nekrashevych--de la Salle \cite{JNS} (see Theorem \ref{t-strongly-bounded}). The class of groups of strongly bounded type includes in particular groups of \emph{bounded automorphisms} of rooted trees \cite{Nek-free} (e.g. groups generated by finite-state \emph{bounded automata}), which contain many well-studied examples of weakly branch groups acting on rooted trees (see \S \ref{s-bounded-automorphisms}), as well as other groups of homeomorphisms of Cantor sets such as topological full groups (see \S \ref{s-bounded-full}). 
In particular every embedding from a finitely generated branch group $G\le \aut(T)$ to a group of strongly bounded type  gives rise to a factor map onto $\partial T$ (Theorem \ref{t-strongly-bounded}).

\subsubsection{The essential support} Let $G\le \aut(T)$ be a branch group, and $X$ a compact $G$-space. Recall that every non-trivial normal subgroup of $G$ contains  $\rist_G(n)'$ for some $n$, by Grigorchuk's  lemma (Lemma \ref{l-Gri-branch}).   For every $n$ the support $\supp_X \rist_G(n)'$ is a closed $G$-invariant subset, and the action  of $G$ on its complement factors via an action of the virtually abelian quotient $G/\rist_G(n)'$. Thus,  it is natural to restrict the attention to the action on $\supp_X \rist_G(n)'$. Since the groups $\rist_G(n)$ form a decreasing sequence,  we have a   sequence of closed $G$-invariant subsets of $X$:
\[\supp_X(\rist_G(n)') \supset \supp_X(\rist_G(n+1)') \supset \supp_X (\rist_G(n+2)')\supset \cdots\]
This leads to the following definition.

\begin{defin}
	Let $G\le \aut(T)$ be a branch group and $X$ be a compact $G$-space. The \textbf{essential support} of the action of $G$ on $X$ is the set 
	\[ Y := \bigcap_{n\ge 1} \supp_X (\rist_G(n)').\]
\end{defin}

Note that the essential support $Y$ is closed and $G$-invariant, and by compactness $Y=\varnothing$ if and only if there exists $n\ge1$ such that $\supp_X (\rist_G(n)')=\varnothing$, if and only if the action of $G$ on $X$ factors via a virtually abelian quotient $G/\rist_G(n)'$. In particular if $Y\neq \varnothing$, then the action of $G$ on $Y$ is faithful.

\subsubsection{Actions with finite-dimensional graphs of germs} We will use the following notion of dimension of a graph. 

\begin{defin} \label{d-leud} For a bounded degree graph $\Gamma$, we define the \textbf{Lipschitz euclidean dimension} of $\Gamma$, denoted $\leud(\Gamma)\in [0, \infty]$, as the supremum over all $d$ such that there exists a Lipschitz map $\N^d \to \Gamma$ whose fibers have uniformly bounded cardinality. 
\end{defin}

\begin{remark}\label{r-leud-pol}
	If a graph $\Gamma$ has volume growth bounded above by a polynomial of degree $d$, then $\leud(\Gamma) \le d$. 
	It is also worth mentioning that  $\leud(\Gamma)$ is bounded above by the asymptotic dimension of $\Gamma$ \cite{Gro-asdim} (although we will not use asymptotic dimension directly). This follows from the fact that $\N^d$ has asymptotic dimension $d$, and from the observation that asymptotic dimension behaves monotonically for Lipschitz maps with uniformly bounded fibers between graphs of bounded degrees (see  \cite[\S 6]{Benj-reg} or \cite[Prop. 2.5]{MB-graph-germs}). 
\end{remark}

The following lemma is straightforward from Definition \ref{d-leud}, and we omit the proof.

\begin{lem} \label{l-leud-monotone} Let $\Gamma$ and $\Delta$ be graphs of bounded degree. Assume that there exists a Lipschitz map $\Gamma\to \Delta$ whose fibers have uniformly bounded cardinality. Then we have $\leud(\Gamma) \le \leud(\Delta)$.
\end{lem}

Our goal in this paragraph is to prove a rigidity result for actions of branch groups on compact spaces whose graphs of germs have finite Lipschitz euclidean dimension (Theorem \ref{t-actions-leud} below).  We first prove the following proposition, which is an application of Theorem \ref{thm-doublecomm-conf-homeo-intro}.

\begin{prop} \label{p-branch-dichotomy} Let $G\le \aut(T)$ be a finitely generated branch group. Let  $X$ be a compact $G$-space, and $Y\subset X$ be the essential support of the action.  Then one of the following holds. 
	
	\begin{enumerate}[label=\roman*)]
		\item \label{i-inj-germ} There exists a vertex $v\in T$ and a point $x\in X$ such that $\rist_G(v)_x^0 =\{1\}$. 
		
		\item There exists an upper semi-continuous $G$-equivariant map $q \colon Y \to \F(\partial T)$, which takes values in the set of non-empty closed subsets of $\partial T$ with empty interior.
	\end{enumerate}
\end{prop}

\begin{proof}
	We assume that $\rist_G(v)_x^0\neq \{1\}$ for every vertex $v$ and every $x\in X$, and we prove that the second condition must hold. This assumption implies that the group $\rist_G(v)_x^0$ is confined in $\rist_G(v)$ for every $x\in X$. Thus, by Theorem \ref{thm-doublecomm-conf-homeo-intro} for every $v$ and every $x$ there exists a vertex $w\in T$ below $v$ such that $\rist_G(w)' \le \rist_G(v)_x^0 \le G_x^0$. In particular,  for every $x\in X$ the set $O_x=\bigcup\{ \partial T_w \colon \rist_G(w)' \le G_x^0\}$ is a dense open subset of $\partial T$. Note that $O_x=\partial T$ if and only if there exists $n$ such that $\rist_G(n)' \le G_x^0$, and this is equivalent to the fact   $x$ does not belong to the essential support $Y$ of the action. For $x\in Y$, we set $q(x):= \partial T \setminus O_x$. Then $q(x)$ is non-empty and has empty interior by construction, and the map $q$ is clearly equivariant. Let us check that it is upper semicontinuous. Let $(x_i)\subset Y$ be a net converging to some point $x\in Y$, and let $C$ be a cluster point of $(q(x_i))$. We shall check that $C \subseteq q(x)$. Assume for a contradiction that there exists $\xi \in C\setminus q(x)$. Then we can find a vertex $w$ above $\xi$ such that $\rist_G(w)' \le G^0_x$. By Theorem \ref{t-Francoeur}, the group $ \rist_G(w)' $ is finitely generated, so it follows that $\rist_G(w)' \le G^0_{x_i}$ eventually. Therefore $q({x_i}) \cap \partial T_w= \varnothing$ eventually, which is a contradiction with the fact that $C$ contains the point $\xi\in \partial T_w$. \qedhere
\end{proof}

Given a  compact space $Z$ and $d\ge 1$ we let $Z^{[d]}$ be the  subspace of $\F(Z)$ consisting of non-empty finite sets of cardinality at most $d$, that is, the space of configurations of $d$ points on $Z$ without taking into account order or multiplicity. Note that $Z^{[1]}$ is naturally identified with $Z$.

The following is the main result of this subsection. Recall that the definition of graphs of germs was given in \S \ref{s-preliminaries}.

\begin{thm}[Rigidity of actions with finite-dimensional graphs of germs] \label{t-actions-leud} Let $G\le \aut(T)$ be a finitely generated branch group. Let $X$ be a compact space, and let $H \le \homeo(X)$ be a finitely generated group such that the quantity
\[ d:= \sup_{x\in X} \leud(\widetilde{\Gamma}(H, x))\]
is finite.  Let   $\rho \colon G \to H$ be an injective homomorphism, and $Y\subset X$ be the essential support of the $G$-action induced by $\rho$.  Then there exists an upper semi-continuous $G$-equivariant map $q\colon Y \to \partial T^{[d]}$. Moreover, if $d=1$, then the map $q \colon Y \to \partial T^{[1]}=\partial T$ is continuous and surjective.
	
	
\end{thm}

\begin{proof}
	We apply Proposition \ref{p-branch-dichotomy}  to the action of $G$ on $X$. We claim that case \ref{i-inj-germ} in Proposition \ref{p-branch-dichotomy} cannot hold here. Indeed, if by contradiction there exist $x\in X$ and  $v\in T$ such that $\rist_G(v)_x^0=\{1\}$,  the Cayley graph of $\rist_G(v)$  (with respect to any finite generating subset of $\rist_G(v)$)  admits a Lipschitz embedding into the graph of germ $\widetilde{\Gamma}(H, x)$.  But $\rist_G(v)$ contains subgroups isomorphic to direct products of $n$ infinite finitely generated subgroups for arbitrary large $n$ (e.g. $\rist_G(w_1)\times \cdots \times \rist_G(w_n)$ where $w_1,\ldots, w_n$ are vertices below $v$ with $\partial T_{w_i} \cap \partial T_{w_j}=\varnothing$). Hence $\leud(\widetilde{\Gamma}(H, x))=\infty$, which contradicts that $d<\infty$. 
	
	We deduce that we have an upper semicontinuous map $q\colon Y\to \F(\partial T)$, where $Y\subset X$ is the essential support of the $G$-action. Let us show that $q$ must take values in $\partial T^{[d]}$. Assume by contradiction that there exists $y\in Y$ such that $q(y)$ contains $d+1$ distinct points $\xi_1,\ldots, \xi_{d+1}$, and let us construct a Lipschitz injective map $\N^{d+1}\to \widetilde{\Gamma}(H, y)$. Choose vertices $w_i$ above $\xi_i$, with $\partial T_{w_i} \cap \partial T_{w_j}=\varnothing$ for  $i\neq j$. For $i=1,\ldots, d+1$ let $S_i$ be a finite symmetric generating set of $\rist_G(w_i)$, and set  $C_i=q(y) \cap \partial T_{w_i}$. Then $C_i$ is a non-empty closed subset of $\partial T_{w_i}$ with empty interior. We consider the action of $\rist_G(w_i)$ on the space $\F(\partial T_{w_i})$, and restrict the attention to the orbit of $C_i$, denote it by $\Omega_i$. We claim that $\Omega_i$ must be infinite. Namely, if $|\Omega_i|<\infty$ the set $\bigcup_{C\in \Omega_i} C$ would be a closed $\rist_G(w_i)$ invariant subset of $\partial T_{w_i}$ of empty interior.  Since $G$ is a branch group,   Lemma \ref{l-minimal-fi} implies that   $\bigcup_{C\in \Omega_i} C$ is clopen, contradicting that $C_i$ has empty interior. This proves the claim. 
	
	Since each $\Omega_i$ is infinite, for every $i$ we can find an infinite sequence $(s^{(i)}_n)_{n\in \N}$ of elements of $S_i$ with the property that the sets $C_i, h^{(i)}_1(C_i),h^{(i)}_2 (C_i),\cdots$, with $h_n^{(i)}:=s_n^{(i)}\cdots s_1^{(i)}$, are pairwise distinct. For $\vec{n}=(n_1,\ldots, n_{d+1})\in \N^{d+1}$ set $h_{\vec{n}}= h^{(1)}_{n_1}\cdots  h^{(d+1)}_{n_{d+1}}$. Consider the map
	\[\iota\colon \N^{d+1} \to \widetilde{\Gamma}(H, y), \quad \iota(\vec{n})=\rho(h_{\vec{n}}) H_y^0.\] 
	
	The map $\iota$ is clearly Lipschitz, and we claim that it is injective. To see this it is enough to show that for $\vec{n} \neq \vec{m}$, we have $\rho(h_{\vec{n}}(y))\neq \rho(h_{\vec{m}})(y)$. And since the map $q\colon Y \to \F(\partial T)$ is equivariant, it is actually enough to show that  $h_{\vec{n}}(y)(q(y))\neq h_{\vec{m}}(q(y))$. To see this assume that $\vec{n}=(n_1,\ldots, n_{d+1})$ and $\vec{m}=(m_1,\ldots, m_{d+1})$ differ at the $i$th coordinate. Then the sets $h_{\vec{n}}(y)(q(y)) \cap \partial T_{w_i}= h^{(i)}_{n_i}(C_i)$ and $ h_{\vec{m}}(y)(q(y)) \cap \partial T_{w_i}= h^{(i)}_{m_i}(C_i)$ are distinct by construction.  This shows the injectivity of $\iota$. The existence of the map $\iota$ as above contradicts that $\leud(\widetilde{\Gamma}(H, y)) \le d$, thus showing that $q$ takes valued in $\partial 
	T^{[d]}$. Finally observe that if $d=1$, upper semi-continuity of the map $q$  automatically implies that $q$ is  continuous. Surjectivity then follows from the fact that the $G$-action on $\partial T$ is minimal, since $q(Y)$ is a closed non-empty invariant subset.  \qedhere
	\end{proof}



\begin{remark}
Theorem \ref{t-actions-leud} applies for instance when the graphs $\widetilde{\Gamma}(H, x)$ have asymptotic dimension $1$, since in that case we have $d=1$ (see Remark \ref{r-leud-pol}). 
\end{remark}

\subsubsection{Actions on the Cantor set of strongly bounded type} \label{s-strongly-bounded}
Theorem \ref{t-actions-leud} is particularly relevant when the graphs of germs of the group $H$ satisfy $\leud(\widetilde{\Gamma}(H, x)) = 1$. Here we describe a class of groups of homeomorphisms of  the Cantor set which appear naturally in the setting of branch groups, and which satisfy this condition. 

Recall that a \textbf{Bratteli diagram} $B$ is the data of two  sequences of finite non-empty sets $(V_n)_{n\ge 0}, (E_n)_{n\ge 1}$ together with surjective maps $\org \colon E_n \to V_{n-1}, \tg\colon E_n \to V_n$. One can visualize $B$ as a graph whose vertex set is $V=\bigsqcup V_n$ and each $e\in E_n$ is an edge between $\org(e)$ and $\tg(e)$. The \textbf{path space} of $B$, denoted $X_B$, is the space of all infinite paths of the form $x=e_1e_2e_3\cdots$, with $e_i\in E_i$ and  $\tg(e_i)=\org(e_{i+1})$. The space $X_B$ is  endowed with the topology inherited from the product topology on $\prod E_n$, which makes it a totally disconnected compact space. A \textbf{finite path} in $B$ is a finite sequence of the form $\gamma=e_1\cdots e_n$, with $\tg(e_i)=\org(e_{i+1})$. 
We extend the map $\tg$ to the set of finite paths by setting $\tg(\gamma)=\tg(e_n)$. For each $v\in \bigsqcup V_n$, we let $\L(v)$ be the set of finite paths $\gamma$ such that $\tg(\gamma)=v$. 

For every finite path $\gamma$ we denote by $C_\gamma \subset X_B$  the cylinder set of infinite paths that begin with $\gamma$. Then $C_\gamma$ is a clopen subset of $X_B$ and sets of this form are a basis of the topology. Assume that $v\in V_n$ and that $\gamma, \eta\in \L(v)$. Then we have a natural homeomorphism 
\[F_{\gamma, \eta}\colon C_\gamma \to C_\eta, \quad \gamma e_{n+1}e_{n+2} \cdots \mapsto \eta e_{n+1}e_{n+2}\cdots,\] 
Maps of this form are called \textbf{prefix replacement} maps. A homeomorphism $f\colon X_B \to X_B$ is said to be \textbf{finitary} if every point $x\in X_B$, the germ of $f$ at $x$ coincides with the germ of  some prefix replacement $F_{\gamma, \eta}$ with $x\in C_\gamma$. Equivalently if there exists partitions $X_B=C_{\gamma_1}\sqcup \cdots \sqcup C_{\gamma_k}= C_{\eta_1} \sqcup \cdots \sqcup C_{\eta_k}$, with $\tg(\gamma_i)=\tg(\eta_i)$, such that $f|_{C_{\gamma_i}}= F_{\gamma_i, \eta_i}$. Finitary homeomorphisms of $X_B$ form a group. This group coincides with the topological full group of the AF-groupoid associated to the Bratteli diagram, and it is not difficult to see that it is locally finite (see e.g. \cite[\S 2]{Mat-simple}). Note also that if $G$ is a group of finitary automorphisms, then the isotropy group $G_x/G^0_x$ is trivial  for every $x\in X_B$, as follows from the fact that a prefix replacement map that fixes $x$ must automatically fix a neighbourhood of $x$. 

The following definition is due to Juschenko, Nekrashevych, and de la Salle \cite[\S 4]{JNS}
\begin{defin} \label{d-bounded-type}
	Let $g$ be a homeomorphism of $X_B$. We say that a point $x\in X_B$ is a \textbf{singularity} of $g$ if the germ of $g$ at $x$ does not coincide with the germ of a prefix replacement map. Furthermore, for every $v\in \bigsqcup V_n$, let $A_v(g)$ be the number of paths $\gamma\in\L(v)$ such that the restriction of $g$ to $C_\gamma$ does not coincide with a prefix replacement. The homeomorphism $g$ is \textbf{of bounded type}  if it has finitely many singularities and $\sup_v A_v(g)<\infty$. 
	\end{defin}

Homeomorphisms of bounded type of $X_B$ form a group. We further give the following variant of the previous definition. 

\begin{defin} \label{d-strongly-bounded-type}
	A group of homeomorphisms $G$ of $X_B$ is said to be of \textbf{strongly bounded type} if it consists of homeomorphisms of bounded type and if for every $x\in X_B$, the isotropy group $G_x/G^0_x$ is locally finite.
	
	A group of homeomorphisms $G$ of a totally disconnected compact space $X$ is said to be of (strongly) bounded type if there exists a homeomorphisms between $X$ and the path space of a Bratteli diagram that conjugates $G$ to a group of homeomorphisms of (strongly) bounded type. 
	
\end{defin}

Let us observe the following. 

\begin{lem} \label{l-isotropy-fg}
	Let $G$ be a finitely generated group of homeomorphisms of bounded type of  $X_B$. Then for every $x\in X_B$, the isotropy group $G_x/G^0_x$ is finitely generated. In particular, $G$ is of strongly bounded type if and only if the isotropy groups are finite.  
\end{lem}
\begin{proof}
	Fix a finite symmetric generating set $S$ of $G$, which will be implicit throughout the proof.  Let  $\Gamma:=\Gamma(H, x)$ be the orbital graph, and consider its fundamental group $\pi_1(\Gamma, x)$. There is a natural epimorphism $\pi_1(\Gamma, x) \to G_x$, which associate to every loop in $\Gamma$ based at $x$ the product $s_n\cdots s_1$ of the generator that label its edges. By composition we obtain an epimorphism $\pi_1(\Gamma, x) \to G_x/G^0_x$.  Call an edge $(y, s)$ of $\Gamma$ \emph{singular} if $y$ is a singularity of $s\in S$ in the sense of Definition \ref{d-bounded-type}. Note that $\Gamma$ has a finite number of singular edges. Let $\Gamma_1, \ldots, \Gamma_k$ be the connected components of the graph obtained from $\Gamma$ by removing all singular edges. We  also fix a connected finite graph $\Delta \subset \Gamma$ large enough so that it contains $x$, it contains all singular edges, and whenever two vertices of $\Delta$ belong to a same component $\Gamma_i$, they are connected in $\Delta$ by a path which is entirely contained in $\Gamma_i$.
	Now consider the  CW 2-complex $\Gamma^{(2)}$ obtained by filling with a  2-cell every loop in $\Gamma$ which does not contain any singular edge.  
	We claim that the epimorphism $\pi_1(\Gamma, x) \to G_x/G^0_x$ factors via the natural quotient $ \pi_1(\Gamma_x) \to \pi_1(\Gamma^{(2)}, x)$ to a homomorphisms $\pi_1(\Gamma^{(2)}, x)\to G_x/G^0_x$.  This follows from the observation that if $\alpha$ is a  a loop in $\Gamma$ based at any point $y$ which does not contain any singular edge, and if  $s_1,\ldots, s_n$ are the labels read on its edges, then the germ of $s_i\ldots s_1$ at $y$ coincides with the germ of a prefix replacement. In particular $s_n\cdots s_1$ acts trivially on a neighbourhood of $y$, and thus defines a trivial element in the isotropy group $G_y/G^0_y$. 
	Now note that  each of the subgraphs $\Gamma_i$ spans a simply connected subcomplex  $\Gamma_i^{(2)}$ of $\Gamma$, and this implies that every loop in $\Gamma$ based at $x$ is homotopic to a loop  contained in $\Delta$. Thus the group $\pi_1(\Delta, x)$ projects surjectively onto $G_x/G^0_x$. Since $\Delta$ is a finite graph, the group  $\pi_1(\Delta, x)$  is a finitely generated free group, and thus $G_x/G^0_x$ is finitely generated. \qedhere
	
\end{proof}

Let $\Gamma$ be a bounded degree graph.   We will say that $\Gamma$ admits a   \textbf{sequence of bounded cut-sets } if there exists an increasing sequence $V_n$ of  subsets  of the vertex sets of $\Gamma$, with $\Gamma=\bigcup_{n}V_n$, such that $\sup_{n} |\partial V_n |< \infty$. Here the boundary $\partial V$ of  subset $V\subset \Gamma$ is defined as the set of $v\in V$ which admit at least one neighbour outside of $V$.

The relevance of this notion in this setting comes from the following proposition, first proven by Bondarenko in his thesis \cite{Bon} for bounded automata groups acting on rooted trees (a class of groups whose definition will be recalled below). It was shown in \cite[Lemma 4.3]{JNS} that it extends to all groups of homeomorphisms of bounded type.  

\begin{prop} \label{p-bond}
	Let $G$ be a finitely generated groups of homeomorphisms of bounded type of $X_B$. Then for every $x\in X_B$, the orbital graph $\Gamma(G, x)$ admits a sequence of bounded cut-sets. 
	
\end{prop}

We will make use of the existence of a sequence of bounded cut-sets through the following lemma.

\begin{lem} \label{l-leud-cutsets}
	Let $\Gamma$ be a bounded degree graph. If $\Gamma$ admits a sequence of bounded cut-sets, then  $\leud(\Gamma) \le 1$.
\end{lem} 

\begin{proof}

More generally, we will show the following: let $\Delta$ be another bounded degree graph, and assume that there exists a Lipschitz map $\iota \colon \Delta \to \Gamma$ whose fibers have uniformly bounded cardinality.  Then $\Delta$ also has a sequence of bounded cut-sets. The statement  follows by choosing   $\Delta=\N^d$, which does not have a sequence of bounded cut-sets unless $d=1$.

	Let $\iota$ have Lipschitz constant $K>0$, and  fibers with cardinality bounded by $C_1>0$.  Let also $m\ge 1$ be an upper  bound on the degree of $\Gamma$, and $(V_n)$ be  a sequence of finite subsets of $\Gamma$ with $C_2:=\sup |\partial V_n|<\infty$. Set $W_n=\iota^{-1}(V_n)$. We  claim that $(W_n)$ is an exhaustion of $\Delta$ with bounded boundary. We have $|W_n|\le C_1|V_n|$, so that the sets $W_n$ are finite, and it is clear that the sets $W_n$ are increasing and that $\bigcup W_n =\Delta$. Thus we only need to bound the size of the boundary $\partial W_n$. Let $w\in \partial W_n$, and $z\in \Delta$ be a neighbour such that $z\notin W_n$. We have $\iota(w)\in V_n, \iota(z)\notin V_n$, and the distance between $\iota(w)$ and $\iota(z)$ is $\le K$. Since any geodesic path from $\iota(w)$ and $\iota(z)$ must contain a vertex $v\in \partial V_n$, we deduce that $\iota(w)$ is at distance $\le K$ from some vertex $v\in \partial V_n$, that is $w$ is in the preimage of the ball $B_\Gamma(v, K)$. Since $w$ is arbitrary we obtain $\partial W_n \subset \bigcup_{v\in \partial V_n} \iota^{-1}(B_{\Gamma}(v, K))$.  But since the degree of $\Gamma$ is bounded by $m$, we have $|B_{\Gamma}(v, K)|\le m^k$ for every $v$. Thus for every $n$ we have
	\[|\partial W_n| \le \sum_{v\in \partial V_n} |\iota^{-1}(B_\Gamma(v, K))|\le |\partial V_n|C_1 m^K  \le C_2C_1m^K,\]
	showing that $|\partial W_n|$ is uniformly bounded. \qedhere

\end{proof}


We refer to \S \ref{s-bounded-automorphisms} and \S \ref{s-bounded-full} for examples of classes of groups to which the following result applies.

\begin{thm}\label{t-strongly-bounded}
	Let $G\le \aut(T)$ be a finitely generated branch group, and $H$ be a group of homeomorphisms of strongly bounded type of a totally disconnected compact space $X$. Let $\rho \colon G\to H$ be an injective homomorphism, and $Y\subset X$ be the essential support of the action induced by $\rho$. Then there exists a continuous surjective $G$-equivariant map $q\colon Y\to \partial T$. 
\end{thm}
\begin{proof}
	The image $\rho(G)$ is a finitely generated group of homeomorphisms of $X$ of strongly bounded type. Thus the orbital graph of every $x\in X$ satisfies $\leud(\Gamma(G, x))=1$ according to Proposition \ref{p-bond} and Lemma \ref{l-leud-cutsets}. Moreover, by Lemma \ref{l-isotropy-fg} the isotropy group $G_x/G_x^0$ is finite. Since the graph of germs $\widetilde{\Gamma}(G, x)$ covers $\Gamma(G, x)$ with fibers of cardinality $|G_x/G^0_x|$, by Lemma \ref{l-leud-monotone} we also have $\leud(\widetilde{\Gamma}(G, x))\le 1$. Thus, we can apply Theorem \ref{t-actions-leud} and deduce the conclusion. \qedhere
\end{proof}

\subsubsection{Groups of bounded automorphisms of rooted trees}  \label{s-bounded-automorphisms}
An important class of groups of homeomorphisms of strongly bounded type appear among groups of automorphisms of rooted trees. Namely, consider the special case of Bratteli diagram such that each set $V_i$ is contains only one point, so that it is determined by the sequence $(E_i)_{i\ge 1}$. We interpret $(E_i)$ as a sequence of finite alphabets, and the set of finite paths consists of formal words $e_1\cdots e_n$ with $e_i\in E_i$. This set naturally has the structure of rooted tree $T$, called the \textbf{tree of words} associated to the sequence $(E_i)$, whose $n$th level $\L(n)$ is consists of words of length $n$, and where each word $w= e_1\cdots e_n$ is connected by an edge to words of the form $we, e\in E_{n+1}$. (The tree $T$ should not be confused with the graph associated to the Bratteli diagram). The boundary $\partial T$ is identified with the path space of the diagram. 

If $T$ is a tree of words, for every $n$ we denote by $T^{(n)}$ the tree of words associated to the shifted sequence $(E_{i-n})_{i\ge {n+1}}$.  For every $v=e_1\cdots e_n$, the subtree $T_v$ is equal to  the set $vT^{(n)}$ of concatenations of the form $vw, w\in T^{(n)}$. For every $g\in G$ and every $v\in \L(n)$, there exists  a unique $g|_v\in \aut(T^{(n)})$ satisfying
\[g(vw)= g(v)g|_v(w), \quad w\in T^{(n)}.\]
The element $g|_v$ is called the \textbf{section} of $G$ at $v\in T$.   

\begin{defin} \label{d-bounded-automorphisms} Following Nekrashevych \cite[Def. 4.3]{Nek-free}, we say that an automorphism $g\in \aut(T)$ is \textbf{bounded} if there exists finitely many points $\xi_1,\ldots, \xi_k\in \partial T$ and a number $m>0$ such that $g|_w$ is trivial for every $w$ which does not belong to the $m$-neighbourhood of the rays defining $\xi_1,\ldots, \xi_k$. 
\end{defin}

It is not difficult to see that bounded automorphisms of a tree of words $T$  form a subgroup of $\aut(T)$ (see  \cite[Prop. 4.3]{Nek-free}), that we denote $\mathcal{B}(T)$. 
\begin{prop}
	Let $T$ be a tree of words of bounded degree. Then the group of  bounded automorphisms $\mathcal{B}(T)$ is a group of homeomorphisms of $\partial T$ of strongly bounded type in the sense of Definition \ref{d-strongly-bounded-type}. 
\end{prop}
\begin{proof}
	It is clear that $\mathcal{B}(T)$ consists of groups of homeomorphisms of bounded type of $\partial T$, and it is not difficult to check that its isotropy groups are locally finite, using that $T$ has bounded degree  (see the argument in the proof of \cite[Th. 4.4.] {Nek-free}, or  \cite[Lemma 2.7]{AAMBV}).
\end{proof}

Thus, from Theorem \ref{t-strongly-bounded} we obtain the following  result.
\begin{cor}\label{c-bounded-automorphisms}
	Let $G\le \aut(T_1)$ be a finitely generated branch group, and let $T_2$ be a tree of words of bounded degree. Let $G\to \mathcal{B}(T_2)$ be an injective homomorphism, and $Y\subset \partial T_2$ the essential support of the associated $G$-action. Then there exists a continuous surjective equivariant map $Y \to \partial T_1$. 
\end{cor}

Many well studied examples of groups acting on rooted trees arise as subgroups of the group of bounded automorphisms $\mathcal{B}(T)$. An important class of subgroups of $\mathcal{B}(T)$ are groups generated by \emph{finite state bounded automata}, defined by Sidki \cite{Sid}. Let us recall this notion. Assume the sequence of alphabets $(E_i)$ is constant equal to some given alphabet $E$. The corresponding tree $T$ is called the \textbf{regular tree of words} associated to $E$. Note that the shifted trees $T^{(n)}$ are all equal to $E$, so that for every $g\in \aut(T)$ and $w\in T$ the section $g|_w$ is still an element of $\aut(T)$. We  say that $g\in \aut(T)$ is \textbf{finite state} if there exists a finite subset $A\subset \aut(T)$ such that $g|_w\in A$ for every $w\in T$.  This is equivalent to the fact that $g$ can be defined by a finite-state automaton over the alphabet $E$. The \textbf{activity function} of $g\in \aut(T)$ is the function $A_g(n)$ that counts the number of vertices $w\in \L(n)$ such that $g|_w$ is non-trivial (compare with Definition \ref{d-bounded-type}). Obviously if $g\in \aut(T)$ is bounded in the sense of Definition \ref{d-bounded-automorphisms}, then its activity function $A_g(n)$ is bounded. One can check that if $g$ is finite state, then the converse also holds. A \textbf{bounded automaton group} is a  group $G\le \aut(T)$ generated by finitely many bounded finite-state automorphisms. This class contains many well-studied examples of (weakly) branch groups, including the first Grigorchuk group, the Gupta-Sidki groups, the  Basilica group, iterated monodromy groups of post-critically finite polynomials \cite{Nek-book}, see \cite{BKN} for more examples. 
Examples of subgroups of $\mathcal{B}(T)$ which are not finite state include the extended family of Grigorchuk groups $(G_\omega)$ \cite{Gri-growth}. See also \cite{Nek-minimal-cantor} and  \cite[\S 4.3.3]{JNS} for additional examples of subgroups of $\mathcal{B}(T)$.
In particular Corollary \ref{c-bounded-automorphisms} applies to homomorphisms between any two groups among the above classes of groups, provided that the source group is a branch group.

\subsubsection{More examples of  actions of branch groups of strongly bounded type} \label{s-bounded-full}
Groups of homeomorphisms of strongly bounded type also appear outside of the realm of groups acting on rooted trees, in particular in the setting of topological full groups of \'etale groupoids. We refer to \cite[\S 4]{JNS} for a list of examples.   In particular the {topological full group} of every minimal homeomorphism of the Cantor set is  a group of strongly bounded type (this follows from the Bratteli-Vershik representation of Cantor minimal systems obtained in \cite{Her-Put-Ska}, see \cite[\S 4]{JNS} for details).   The family of groups constructed by Nekrashevych in \cite{Nek-frag}, some of which are simple and have intermediate growth, are also all groups of homeomorphisms of strongly bounded type.  Thus Theorem \ref{t-strongly-bounded} applies to homomorphisms from branch groups with values in these classes of groups. Explicit examples of embeddings of Grigorchuk groups into groups in these classes  can be found in  \cite{MB-Gri-full} and \cite {Gri-Le-Na} (see also the tightly related construction in \cite{Vor}) and \cite{Nek-frag}. More examples of homomorphisms from branch groups to groups of homeomorphisms of strongly bounded type can be found in \cite{MB-full-automata}.

\subsubsection{Contracting actions on rooted trees} Theorem \ref{t-actions-leud} also implies a rigidity result of  another type of actions of branch groups on rooted trees, namely contracting actions. 

\begin{defin}
	Let $T$ be a regular tree of words associated to a finite alphabet $E$. A subgroup $G\le \aut(T)$ is \textbf{contracting} if it consists of finite state automorphisms and if there exists a finite set $\mathcal{N}\subset \aut(G)$ such that for every $g\in G$, there exists $m>0$ such that $g|_v\in \mathcal{N}$ for every $v\in T$ at level $\ge m$.  
	
	We say that an action of a group $G$ on $T$ is contracting if it comes from a homomorphism $G\to \aut(T)$ taking values in a contracting group. 
\end{defin}
Contracting groups constitute the central objects in  the theory of iterated monodromy groups, see \cite{Nek-book}. Groups generated by finite-state bounded automata (see \S \ref{s-bounded-automorphisms}) are contracting, see \cite[Th. 3.8.8]{Nek-book}, but the class of contracting groups is strictly larger than the class of bounded automata groups.

\begin{remark} \label{r-contracting-not-unique}
	If $G$ is a contracting group, a contracting action of $G$ on $T$ is in general not unique, even up to conjugation on the boundary. For example, given a tree of words $T$ over the alphabet $E$ for every $d\ge 1$, consider the tree of words $T^{\otimes d}$ associated to the sequence of alphabets $E^d$. The group $G$ acts on $T^{\otimes d}$, and this action is contracting  if the original action of $G$ on $T$ was. Note that the boundary $\partial T^{\otimes d}$ identifies with $(\partial T)^d$, with the diagonal action of $G$. 
	
\end{remark}
Theorem \ref{t-actions-leud} has the following corollary, which can be seen as a partial converse to Remark \ref{r-contracting-not-unique}.

\begin{cor}
	Let $G\le \aut(T_1)$ be a finitely generated branch group, and suppose that $G$ has a faithful contracting action on a regular tree of words $T_2$,  with essential support $Y\subset \partial T_2$. Then there exists $d\ge 1$ and an upper semi-continuous $G$-equivariant map $q\colon Y \to \partial T_2^{[d]}$. Moreover, if the action of $G$ on $T_2$ is level-transitive, then $Y=\partial T_2$ and the map $q\colon \partial T_2 \to \partial T_1^{[d]}$ is continuous.

\end{cor}
\begin{proof}
By a result of Nekrashevych  \cite[Prop. 2.13.6]{Nek-book}, if $G\le \aut(T_2)$ is a finitely generated contracting group, then there exists $d$ such that for every $\xi \in \partial T_2$ the orbital graph ${\Gamma}(G, \xi)$ has  polynomial growth of degree at most $d$. Moreover by \cite[Prop. 4.1]{Nek-free}, the isotropy groups $G_\xi/G^0_\xi$ are finite and their cardinality is uniformly bounded in $\xi$. Thus, the graphs of germs $\widetilde{\Gamma}(G, \xi)$ also have  polynomial growth of degree at most $d$. By Remark \ref{r-leud-pol}, we deduce that   $\leud(\widetilde{\Gamma}(G, \xi))\le d$ for every $\xi \in \partial T_2$. Thus Theorem \ref{t-actions-leud} applies. Assume further that the action of $G$ on $T_2$ is level-transitive, i.e. the action on $\partial T_2$ is minimal. Since $Y$ is closed and $G$-invariant, we must have $Y=\partial T_2$. Upon replacing $d$ with a smaller number, we can assume that the image of the map $q\colon \partial T_2 \to \partial T_1^{[d]}$ contains a finite subset $Q=\{\xi_1,\ldots, \xi_d\} \subset \partial T_2$ of cardinality $d$. Let $X\subset \partial T_1^{[d]}$ be the closure of the $G$-orbit of $Q$. Then $X$ consists entirely of sets of cardinality $d$. Since $q$ is upper semicontinuous and $d$ is the maximal cardinality of subsets in $\partial T_1^{[d]}$, the preimage $q^{-1}(X)$ is closed, and since it is also non-empty by minimality we must have $q^{-1}(X)=\partial T_2$. So $q$ takes values in $X$, and using again that the cardinality of sets in $X$ is maximal, the upper semi-continuity forces $q$ to be continuous. \qedhere
\end{proof}

\bibliographystyle{amsalpha}
\bibliography{bib-prox}

\end{document}